\documentclass[11pt,letterpaper]{article}
\usepackage{fix-cm}
\usepackage[T1]{fontenc}
\usepackage{newtxtext}
\usepackage{amsmath,amssymb,amsthm,mathtools,bm}
\usepackage{newtxmath}
\usepackage[margin=1in]{geometry}
\usepackage{microtype,booktabs,tabularx,array,enumitem,graphicx}
\usepackage[round,authoryear]{natbib}
\usepackage{algorithm,algpseudocode}
\usepackage[hidelinks]{hyperref}
\usepackage{fancyhdr}
\hypersetup{pdftitle={The Minimax Cost of Smoothness in Fully First-Order Stochastic Bilevel Optimization},pdfauthor={Anonymous authors}}
\setlist[enumerate]{leftmargin=*,itemsep=2pt,topsep=3pt}
\newtheorem{theorem}{Theorem}[section]
\newtheorem{lemma}[theorem]{Lemma}
\newtheorem{proposition}[theorem]{Proposition}
\newtheorem{corollary}[theorem]{Corollary}
\newtheorem{assumption}{Assumption}
\theoremstyle{definition}\newtheorem{definition}[theorem]{Definition}
\theoremstyle{remark}
\newcommand{\R}{\mathbb R}
\newcommand{\N}{\mathbb N}
\newcommand{\E}{\mathbb E}
\newcommand{\Prob}{\mathbb P}
\newcommand{\norm}[1]{\left\|#1\right\|}

\newcommand{\ip}[2]{\left\langle #1,#2\right\rangle}
\newcommand{\Lip}{\operatorname{Lip}}

\newcommand{\diag}{\operatorname{diag}}
\newcommand{\eps}{\epsilon}
\newcommand{\cF}{\mathcal F}
\newcommand{\cO}{\mathcal O}

\newcommand{\D}{\mathcal D}
\newcommand{\Bern}{\operatorname{Bernoulli}}
\newcommand{\dir}{\operatorname{dir}}
\newcolumntype{Y}{>{\raggedright\arraybackslash}X}
\DeclareMathOperator*{\argmin}{arg\,min}
\allowdisplaybreaks[1]

\title{The Minimax Cost of Smoothness in Fully First-Order Stochastic Bilevel Optimization}
\newcommand{\authornames}{Wendao Wu, Haihan Zhang, Chenheng Zhang, Yanyi Li, Chunyuan Zheng, Cong Fang, Haoxuan Li, Zhouchen Lin}
\newcommand{\affiliation}{Peking University}

\newcommand{\mail}[1]{\href{mailto:#1}{\texttt{#1}}}
\newif\ifanonymous
\anonymousfalse
\ifanonymous
\author{Anonymous Authors}
\hypersetup{pdfauthor={Anonymous Authors}}
\else
\author{Wendao Wu$^{1,*}$\quad Haihan Zhang$^{1,*}$\quad Chenheng Zhang$^{1,*}$\\[0.20em]
Yanyi Li$^{1,*}$\quad Chunyuan Zheng$^1$\\[0.20em]
Cong Fang$^{1,\dagger}$\quad Haoxuan Li$^{1,\dagger}$\quad Zhouchen Lin$^{1,\dagger}$\\[0.60em]
\small $^1$\affiliation\\[0.35em]
\footnotesize\mail{colombo.2026@outlook.com}\\[-0.05em]
\small $^*$Equal contribution.\quad $^\dagger$Corresponding authors.}
\hypersetup{pdfauthor={\authornames}}
\fi

\date{}

\begin{document}
\maketitle
\begin{abstract}
We study the sample complexity of finding a point with expected hypergradient norm at most $\eps>0$ in nonconvex--strongly-convex stochastic bilevel optimization. Algorithms access only fresh, globally unbiased first-order samples with bounded variance. For lower-variable smoothness order $p\ge1$, F$^2$SA-$p$ achieves the upper bound $\widetilde O(p\eps^{-4-2/p})$ \citep{chen2026}, but whether the exponent $4+2/p$ is necessary remains unresolved. We establish matching upper and lower bounds $\Theta_p(\eps^{-4-2/p})$ under fixed nondegenerate regularity budgets, with constants allowed to depend on $p$. The lower bound applies to unrestricted randomized algorithms under a fixed sample budget and allows dimension to grow with accuracy. Our construction hides successive descent directions while preserving global unbiasedness and all prescribed lower-variable smoothness bounds, even with a scalar lower variable and exact upper gradients. We further characterize the minimax fixed sample budget $Q_{p,\beta}^*(\eps)$ under derivative Lipschitz budgets $L_j=M\Lambda^j(j!)^\beta$ for $j\ge3$, with $M,\Lambda>0$ and $\beta\ge0$ fixed independently of order and accuracy:
\[
 Q_{p,\beta}^*(\eps)=\Theta\!\left(\eps^{-4}\left[\inf_{\substack{r\in\N\\1\le r\le p}}r^\beta\eps^{-1/r}\right]^2\right),\qquad p\in\N\cup\{\infty\}.
\]
Here $r$ is a positive integer differentiation order, $p=\infty$ imposes the envelope at every order, and the constants are independent of $p$. At infinite order, this gives $\Theta(\eps^{-4})$ for the exponential envelope ($\beta=0$) and $\Theta(\eps^{-4}\log^2(1/\eps))$ for the analytic envelope ($\beta=1$). A gradient-difference method with weighted independent batches attains these bounds by controlling differentiation and inner-solve costs uniformly in order. These results identify the precise benefit of quantitative smoothness, whereas qualitative analyticity alone still permits complexity $\Theta(\eps^{-6})$.
\end{abstract}
\par

\begin{samepage}
\noindent\textbf{AI Usage.}
Nearly the entire research pipeline for this paper was carried out by Colombo, our laboratory’s internal auto-research system, powered by GPT-5.6 Sol. The system also conducted a Lean-backed article audit of the resulting manuscript. The authors subsequently reviewed and approved the mathematical claims, presentation, and formal artifacts, and take responsibility for the final manuscript. The complete Lean audit report and the system’s technical report will be made public at a later date.  \par
\end{samepage}

\paragraph{Publication Intent and Recognition of Independent Work.} 
This preprint is released as a research output of our study of
automated scientific research, for public scrutiny and discussion.
As we do not intend to submit this manuscript to a journal or
conference, it has not undergone the full editorial refinement
and internal review we would normally require for submission.
The authors remain responsible for its contents and for correcting
any errors or inaccurate attribution.

We recognize that releasing these results may affect researchers
who have independently been pursuing the same questions, including
work that is not yet public.
If this applies to your work, we invite you to contact us at
\texttt{colombo.2026@outlook.com}.
A brief description of the overlap, together with any dated
materials you are comfortable sharing, would help us understand
its scope and chronology.
Such materials might include draft excerpts, presentation slides,
or revision records; a complete manuscript or unpublished proof
is not required to begin the conversation.

We will review these cases carefully and update citations,
acknowledgments, and accounts of independent or concurrent work
as appropriate, including revising priority claims where warranted.
Nonpublic materials will be treated confidentially and used only
to assess the overlap unless their authors authorize otherwise.
We will obtain explicit permission from the authors before
publicly discussing such materials or providing them to an
automated research system.

\section{Introduction}\label{sec:introduction}
We study the sample complexity of nonconvex bilevel objectives with a strongly convex lower problem:
\begin{equation}\label{eq:bilevel}
 \Phi(x)=f(x,y^*(x)),\qquad y^*(x)=\argmin_y g(x,y).
\end{equation}
Here $x\in\mathbb R^{d_x}$ and $y\in\mathbb R^{d_y}$ are the upper and lower decision variables, $f$ and $g$ are the upper and lower objectives, and $y^*(x)$ is the unique lower minimizer; $\Phi$ is the resulting upper objective. The objective $\Phi$ need not be convex. We seek a point whose expected gradient norm is at most a target accuracy $\eps>0$ and count the gradient samples needed to find it. Even when $g(x,\cdot)$ is strongly convex, observing gradients of $f$ and $g$ is not the same as observing $\nabla\Phi$: the latter depends on the sensitivity of the lower minimizer. Fully first-order methods address this distinction without receiving Hessians or Hessian--vector products. 

Penalty methods make this access model algorithmically useful. The original F$^2$SA method gives a fully stochastic $\widetilde O(\eps^{-7})$ guarantee in the gradient-norm convention \citep{kwon2023}, and subsequent methods achieve exponent six \citep{kwon2024,chen2025}. For an integer smoothness order $p\ge1$, F$^2$SA-$p$ exploits higher lower-variable smoothness through finite differences of perturbed lower problems, obtaining $\widetilde O(p\eps^{-4-2/p})$ \citep{chen2026}. Finite-difference adjoint methods offer another route to sensitivity estimation, but their faster stochastic guarantees can rely on samplewise regularity and reused randomness \citep{yang2023fde}. These access differences matter for a minimax comparison: population smoothness alone does not give cancellation between independently sampled gradients.

There are two different complexity questions. At a \emph{fixed} finite smoothness order $p$, existing upper bounds have exponent $4+2/p$, while a separable single-level embedding gives only the lower exponent four \citep{chen2026,arjevani2023}. When $p$ grows with accuracy, even a matching fixed-order exponent is insufficient: its hidden constants may grow arbitrarily with $p$. An order-independent normalization of the derivative sequence is needed.

We resolve both questions for fresh, globally unbiased gradient samples of bounded variance. An algorithm may query anywhere and use arbitrary computation and randomization, but a sampled loss cannot be re-evaluated at a different point. The target is $\E\norm{\nabla\Phi(\widehat x)}\le\eps$, where $\widehat x$ is the algorithm output and the expectation is over oracle and algorithm randomness, under a worst-case fixed sample budget. For arbitrary positive fixed finite regularity budgets, we prove $\Theta_p(\eps^{-4-2/p})$, including the baseline $\Theta(\eps^{-6})$. Let $L_j$ denote the prescribed Lipschitz budgets for lower-variable derivatives of the gradient fields, with the precise indexing in Section~\ref{sec:model}. Under the common envelope $L_j=M\Lambda^j(j!)^\beta$ for integers $j\ge3$, where $M,\Lambda>0$ and $\beta\ge0$ are fixed independently of order and accuracy, we further obtain
\begin{equation}\label{eq:intro-law}
 \Theta\!\left(\eps^{-4}\D_{p,\beta}(\eps)^2\right),
 \qquad \D_{p,\beta}(\eps)=\inf_{1\le r\le p}r^\beta\eps^{-1/r},
\end{equation}
where $\D_{p,\beta}(\eps)$ is the optimized differentiation factor, $r$ ranges over positive integers, and $p=\infty$ removes the upper restriction on $r$. The constants are uniform in finite or infinite $p$. A fixed analytic envelope ($\beta=1$) leaves a necessary and sufficient $\log^2(1/\eps)$ overhead. An infinite exponential envelope ($\beta=0$, $p=\infty$) removes this overhead, whereas qualitative analyticity in the lower variable alone still permits exponent six.

\paragraph{Lower-bound contribution.}
A long chain accumulates descent over many coordinates, but placing its entire value inside a narrow lower perturbation jeopardizes global regularity. A naive quadratic lift also lets distant lower-variable queries amplify the hidden gradient. We separate extensive descent from a bounded frontier, embedding only the frontier through a sine. This keeps the full hidden gradient globally small and permits simultaneous control of all prescribed derivatives. A deferred-prefix random-rotation argument covers unrestricted algorithms. The lower problem is scalar, and exact upper gradients already suffice for every lower bound.

\paragraph{Upper-bound contribution.}
Chebyshev gradient differences and a projected stochastic adjoint attain~\eqref{eq:intro-law}. Polynomial differentiation and projected adjoints are established tools \citep{fornberg1988,yang2023fde}. Here, order-controlled differentiation weights, independent batches allocated across nodes, and geometric batches across contraction steps yield the uniform bound with every sample charged. Neither a multiplicative stencil-size factor nor an inner-loop logarithm is hidden in the result.

Our numerical study separates the two theoretical mechanisms. Budget-dependent sine-lift instances keep all tested methods above the certified expected-error barrier, whereas a fixed short chain permits eventual progress when the budget increases. A separate fixed collection tests the sample-cost growth of Ours, F$^2$SA-$p$, and Kwon's method under a common error criterion. These experiments use practical schedules specified in Appendix~\ref{app:experiments} and illustrate the two mechanisms separately.

This work was developed independently, and an initial manuscript was completed on September 18, 2026. Two concurrent preprints address closely related questions. Gu, Wu, and Yang \citep{gu2026firstorderlower} (first submitted to arXiv on September 18, 2026) establish a stochastic first-order lower bound of $\Omega(\Delta\sigma^2\kappa_y^8\epsilon^{-6})$ in the noise-dominated regime, identifying the optimal baseline exponent in $\epsilon$ and an explicit dependence on the lower-level condition number. Pan and Yang \citep{pan2026optimalbilevel} (first submitted to arXiv on October 1, 2026) develop a single-loop first-order method and establish matching complexity bounds of order $\epsilon^{-4-2/p}$ for each fixed finite smoothness order $p$. Their fixed-order complexity exponent coincides with ours, and their lower-bound construction also uses a scalar lower-level variable with sinusoidal coupling. Beyond this shared fixed-order characterization, our main additional contribution is a minimax law that is uniform in the smoothness order under prescribed derivative-growth envelopes, with constants independent of $p$. Our method attains this law through order-controlled gradient differences and weighted independent sampling, allowing the differentiation order to increase with accuracy. The resulting characterization also determines the infinite-order limits: $\Theta(\epsilon^{-4})$ under an exponential derivative envelope and $\Theta(\epsilon^{-4}\log^2(1/\epsilon))$ under an analytic envelope. We further distinguish these quantitative guarantees from qualitative analyticity alone, which still permits minimax complexity $\Theta(\epsilon^{-6})$.\footnote{Posting our manuscript to arXiv has been delayed by submission limits affecting the lead authors. The manuscript is available through our \href{https://www.overleaf.com/read/mhyxqwzrtnyx\#90cfd6}{Overleaf project}. Readers wishing to verify its development timeline are welcome to contact the authors by email to request reviewer access and the relevant dated revision records if these are unavailable through the public read-only link.}

\subsection{Related work}\label{sec:related}
\paragraph{Fully first-order upper bounds.}
\citet{kwon2023} introduced F$^2$SA. Its fully stochastic $\widetilde O(\varepsilon^{-7/2})$ bound uses a squared-gradient tolerance; setting $\varepsilon=\eps^2$ gives $\widetilde O(\eps^{-7})$ for our target. Later upper bounds attain exponent six \citep{kwon2024,chen2025}. These works also impose different regularity or study other noise regimes. For example, the full Hessian-Lipschitz assumption of \citet[Assumption 3.1]{chen2025} is stronger than our two block conditions. The matching claim here concerns the two-block regularity class defined in Section~\ref{sec:model}.

\paragraph{The closest finite-order comparison.}
\citet{chen2026} interpret penalty gradients as finite differences and develop F$^2$SA-$p$. Their Theorem~3.1 gives, with its other parameters fixed, $O(p\eps^{-4-2/p}\log(C/\eps))$, where $C>0$ depends on the fixed problem parameters, customarily written $\widetilde O(p\eps^{-4-2/p})$, with lower exponent four in Theorem~4.1. The logarithm comes from the inner accuracy schedule in that guarantee; it is not a lower bound on the algorithm\textquotesingle s actual cost. Corollary~\ref{thm:fixed} supplies the matching exponent for the lower-variable framework defined in Section~\ref{sec:model}. The uniform theorem additionally controls derivative growth and the residual logarithms. In their notation, $\bar L=\max_{0\le j\le p}L_j$ is the largest prescribed regularity bound and $\kappa=\bar L/\mu$ is the condition number, with $\mu>0$ the lower strong-convexity modulus. These quantities generally grow after inserting a factorial envelope, making their order dependence relevant to this comparison. Their nearly single-level high-smoothness upper bound is consistent with our exponential-envelope endpoint. The new distinction is the exact normalization-sensitive law and its matching lower bound.

\paragraph{Adjoints, numerical differences, and sample regularity.}
FdeHBO already combines finite-difference Hessian/Jacobian-vector approximations with a projected adjoint \citep{yang2023fde}. Its paired differences and momentum estimators reuse randomness, and its assumptions include samplewise smoothness and stochastic Hessian bounds. Our architecture is related, but our estimates use independently sampled population gradients without those assumptions. The new estimates concern uniform weight norms and fully counted allocations. Polynomial differentiation is classical \citep{fornberg1988}, and the projected-adjoint architecture is already present in FdeHBO.

\begin{table}[t]
\centering\small
\setlength{\tabcolsep}{4pt}
\caption{Selected sample-complexity statements for expected gradient norm at most $\eps$. The column $p$ records the lower-variable smoothness order; $p=1$ denotes the baseline order, and a dash indicates that no order parameter is studied. Oracle and regularity qualifications differ across rows.}\label{tab:literature}
\begin{tabularx}{\textwidth}{@{}p{0.19\textwidth}p{0.10\textwidth}p{0.30\textwidth}Y@{}}
\toprule
Work & Order $p$ & Complexity & Essential qualification\\
\midrule
\citet{kwon2023} & $1$ & Upper $\widetilde O(\eps^{-7})$ & Original F$^2$SA; additional regularity.\\[3pt]
\citet{kwon2024} & $1$ & Upper $O(\eps^{-6})$; lower $\Omega(\eps^{-6})$ & Standard-gradient upper; locally unbiased $y^*$-aware lower.\\[3pt]
\citet{chen2025} & $1$ & Upper $\widetilde O(\eps^{-6})$ & Fully stochastic branch; full Hessian smoothness.\\[3pt]
\citet{chen2026} & $p\in\N$ & Upper $\widetilde O(p\eps^{-4-2/p})$; lower $\Omega(\eps^{-4})$ & Explicit node factor $p$; other constants depend on derivative budgets and order.\\[3pt]
\citet{chenji2026} & --- & SFO lower $\Omega(\eps^{-4})$ & Fixed condition number; constrained zero-respecting class.\\[3pt]
This work, fixed list & $p\in\N$ & $\Theta_p(\eps^{-4-2/p})$ & Fixed order and arbitrary finite positive budgets; unrestricted algorithms.\\[3pt]
This work, envelope & $\N\cup\{\infty\}$ & $\Theta(\eps^{-4}\D_{p,\beta}(\eps)^2)$ & Constants uniform in $p$ under the common derivative envelope.\\
\bottomrule
\end{tabularx}
\end{table}

\paragraph{Lower bounds and other interfaces.}
The $\eps^{-6}$ lower bound of \citet{kwon2024} permits gradients that are unbiased only near $y^*(x)$, while also providing an approximate lower minimizer. Our lower construction instead maintains global unbiasedness at arbitrary queries through a globally bounded full-gradient perturbation. Condition-number bounds instead sharpen another axis \citep{ji2026,chenji2026}; the latter's revised SFO result has stochastic term $\Omega(\kappa_y^{9/2}\eps^{-4})$, where $\kappa_y$ is the lower-level condition number, for constrained, zero-respecting algorithms. Our characterization fixes these structural parameters and varies accuracy and lower-variable regularity. SPABA uses variance reduction and a different derivative/sampling interface \citep{chu2024}. The recent Stoc-SGHA method obtains stochastic exponent six, and exponent four with additional lower-level stochastic smoothness; its expectation guarantees also assume bounded iterates \citep{gu2026}. Robust chains, random rotations, and Bernoulli revelation are established lower-bound ingredients \citep{carmon2020,arjevani2023}; our contribution is their globally regular bilevel realization and order-uniform scaling. Appendix~\ref{app:comparison} records the detailed interfaces and version-specific comparisons.

\paragraph{Organization.}
Section~\ref{sec:model} defines the function classes and oracle. Section~\ref{sec:results} states the minimax characterization. Sections~\ref{sec:upper-overview} and~\ref{sec:lower-overview} explain the two proof mechanisms. Section~\ref{sec:experiments} presents the numerical illustrations, followed by the consequences and limitations in Section~\ref{sec:discussion}. The appendices provide the proofs and experimental specifications.

\section{Model and Quantitative Smoothness}\label{sec:model}
Let $\N=\{1,2,\ldots\}$. All vector norms are Euclidean, and matrix and multilinear tensor norms are induced operator norms. Write $g_{xy}=D_y\nabla_xg$ and $g_{yy}=D_y\nabla_yg$. The notation $\Lip_y(T)$ means Lipschitzness in $y$, uniformly over $x$. The upper variable may be nonconvex. Fix
\begin{equation}\label{eq:base-budgets}
 \mu,L_0,L_2,\Delta,\sigma>0,\qquad L_1>\mu,\qquad R_0\ge0.
\end{equation}
\begin{assumption}[Baseline population regularity]\label{ass:population}
Globally on $\R^{d_x}\times\R^{d_y}$, $g(x,\cdot)$ is $\mu$-strongly convex, $f(x,\cdot)$ is $L_0$-Lipschitz, and both full gradients are jointly $L_1$-Lipschitz. The blocks $g_{xy},g_{yy}$ exist and are jointly $L_2$-Lipschitz. At the supplied starting pair,
\[
 \Phi(x_0)-\inf\Phi\le\Delta,\qquad \norm{y_0-y^*(x_0)}\le R_0.
\]
\end{assumption}
For a fixed integer $p\ge1$, the additional finite-order assumptions are
\begin{equation}\label{eq:high}
 \Lip_y(D_y^q\nabla f)\le L_{q+1},\qquad
 \Lip_y(D_y^{q+1}\nabla g)\le L_{q+2},\qquad 1\le q\le p-1.
\end{equation}
The list is empty when $p=1$. Its bounds concern lower-variable derivatives, not arbitrary high derivatives in $x$. The stated derivatives exist, and Lipschitz top derivatives are interpreted through finite $C^{k,1}$ Taylor remainders; an extra continuous derivative beyond the list is not required. Strengthening the first inequality in~\eqref{eq:high} to joint Lipschitzness does not change our conclusions: the upper proof needs only the displayed bounds, and the hard instances have constant $f_y$ and vanishing positive-order lower derivatives of $\nabla f$.

\begin{definition}[Uniform factorial-power envelope]\label{def:profile}
Fix $M,\Lambda>0$ and $\beta\ge0$, independently of both order and accuracy, and set
\begin{equation}\label{eq:profile}
 L_j=M\Lambda^j(j!)^\beta,\qquad j\ge3.
\end{equation}
For $p\in\N\cup\{\infty\}$, let $\cF_{p,\beta}$ satisfy Assumption~\ref{ass:population} and~\eqref{eq:high} with~\eqref{eq:profile}. For $p=\infty$, all positive integers $q$ are included. The baseline constant $L_2$ remains separately fixed.
\end{definition}
An infinite envelope with $\beta=1$ is quantitatively analytic in $y$: dividing its derivative bounds by the Taylor factorial yields a common local convergence radius. For $\beta>1$ the envelope is of Gevrey type; $0\le\beta<1$ is a stronger factorial-growth restriction. These labels always refer to the derivative fields in~\eqref{eq:high} and do not impose joint analyticity in $(x,y)$.

\begin{assumption}[Fresh first-order observations]\label{ass:oracle}
At an adaptively chosen $(x,y)$, a call selects $a\in\{f,g\}$ and returns a full gradient sample $G_a$ with, conditionally on the past and the query,
\begin{equation}\label{eq:oracle-model}
 \E G_a=\nabla a(x,y),\qquad
 \E\norm{G_a-\nabla a(x,y)}^2\le\sigma^2.
\end{equation}
Independent fresh copies are available; a sample is evaluated only once. No function values, Hessians, Hessian--vector products, exact lower solutions, or exact hypergradients are supplied. Individual sample losses need not have uniform smoothness or convexity. Every sample in every batch is charged.
\end{assumption}
The guarantees in~\eqref{eq:oracle-model} hold on the \emph{entire} domain. Receiving both full gradients costs two calls. Let $Q^*_{p,\beta}(\eps)$ be the least integer $N$ for which an algorithm, informed of the budgets, order, and dimensions, uses at most $N$ samples on \emph{every} execution and satisfies
\begin{equation}\label{eq:target}
 \sup_{d_x,d_y}\ \sup_{(f,g,\cO)\in\cF_{p,\beta}\times\text{admissible oracles}}
 \E\norm{\nabla\Phi(\widehat x)}\le\eps.
\end{equation}
The second supremum means that the oracle obeys~\eqref{eq:oracle-model} for the chosen functions. Computation, memory, internal randomness, and output points are unrestricted. Dimensions may grow with accuracy. The criterion uses expected \emph{norm} and a \emph{fixed} sample budget; squared-norm tolerances and expected stopping costs define different guarantees.

\section{Main Results: Finite Order to Analytic Regularity}\label{sec:results}
Define the order--accuracy tradeoff
\begin{equation}\label{eq:Ddefinition}
 \D_{p,\beta}(\eps)=\inf_{\substack{r\in\N,\ 1\le r\le p}}r^\beta\eps^{-1/r}.
\end{equation}
When $p=\infty$ the upper restriction is omitted. For $\beta=0,p=\infty$, the infimum equals one and is not attained by a finite order.

\begin{theorem}[Uniform-order minimax complexity]\label{thm:uniform}
Under~\eqref{eq:base-budgets} and Definition~\ref{def:profile}, there exist $c,C,\eps_0>0$, depending only on the fixed budgets and $(M,\Lambda,\beta)$, such that \emph{simultaneously} for every $p\in\N\cup\{\infty\}$ and $0<\eps<\eps_0$,
\begin{equation}\label{eq:uniform-law}
 c\eps^{-4}\D_{p,\beta}(\eps)^2\le Q^*_{p,\beta}(\eps)
 \le C\eps^{-4}\D_{p,\beta}(\eps)^2.
\end{equation}
The constants and accuracy threshold are independent of $p$ and dimension. The lower bound holds against unrestricted adaptive randomized algorithms, already with $d_y=1$, exact $f$ gradients, and a globally unbiased $g$ oracle that samples the gradients of two fixed smooth losses. Its initial inner error is zero, and
\[
 d_x=O\!\left(\eps^{-2}\log(1/\eps)\right)
\]
suffices, with an implied constant independent of $p$.
\end{theorem}
The upper half is proved in Appendices~\ref{app:differencing}--\ref{app:upper}; the lower half in Appendices~\ref{sec:split}--\ref{sec:random}. Section~\ref{sec:lower-overview} explains why globally valid gradients do not destroy the information bottleneck.

For comparison with finite-order assumptions not chosen from an envelope, denote by $Q_p^*(\eps;L_3,\ldots,L_{p+1})$ the complexity under an arbitrary fixed positive list in~\eqref{eq:high} and the same baseline budgets.
\begin{corollary}[Arbitrary fixed finite budgets]\label{thm:fixed}\label{thm:main}
For every fixed $p\ge1$ and every such finite positive list,
\begin{equation}\label{eq:fixed-law}
 Q_p^*(\eps;L_3,\ldots,L_{p+1})=\Theta_p(\eps^{-4-2/p})
\end{equation}
for all sufficiently small $\eps$. Constants and the threshold in this statement may depend on that fixed list and on $p$. The same rate holds if all $f$ gradients are exact.
\end{corollary}
This follows by sandwiching the list between two constant envelopes ($\beta=0,\Lambda=1$); Appendix~\ref{app:consequences} proves the class inclusions explicitly. It preserves the arbitrary-budget fixed-order result without substituting a growing $p$ into constants allowed to depend on $p$.

\begin{corollary}[Growing order and the infinite-order endpoints]\label{cor:infty}
Put $t=\log(1/\eps)$. For fixed $\beta>0$ and sufficiently small $\eps$, uniformly in $p$,
\begin{equation}\label{eq:phase}
 Q^*_{p,\beta}(\eps)=
 \begin{cases}
 \Theta_\beta(p^{2\beta}\eps^{-4-2/p}),&p\le t/\beta,\\
 \Theta_\beta(\eps^{-4}t^{2\beta}),&p\ge t/\beta.
 \end{cases}
\end{equation}
For $\beta=0$, $Q^*_{p,0}(\eps)=\Theta(\eps^{-4-2/p})$ uniformly over finite $p$, and $Q^*_{\infty,0}(\eps)=\Theta(\eps^{-4})$. In particular,
\begin{equation}\label{eq:analytic-law}
 Q^*_{\infty,1}(\eps)=\Theta\!\left(\eps^{-4}\log^2(1/\eps)\right).
\end{equation}
All these lower bounds remain valid with exact upper gradients.
\end{corollary}
Indeed, $\beta\log r+t/r$ is minimized continuously at $r=t/\beta$, and integer rounding changes the optimum only by a factor depending on $\beta$. The algorithm always uses a finite $r=O_\beta(t)$, including at $p=\infty$.

Table~\ref{tab:regimes} in Appendix~\ref{app:consequences} summarizes the regimes. In particular, qualitative $C^\infty$ regularity, even entire dependence on $y$ without a common quantitative envelope, still has complexity $\Theta(\eps^{-6})$. Thus~\eqref{eq:analytic-law} concerns the specified derivative budgets, not every class called ``analytic.''

\section{An Order-Uniform, Fully Counted Upper Bound}\label{sec:upper-overview}
The upper method estimates the sensitivity term directly, rather than solving one perturbed inner problem for every difference node. Write
\begin{equation}\label{eq:main-hyper}
 \nabla\Phi(x)=f_x-g_{xy}v^*,\qquad v^*=g_{yy}^{-1}f_y,
\end{equation}
with all quantities evaluated at $(x,y^*(x))$. These derivatives define analysis objects; the algorithm receives only the gradients in Assumption~\ref{ass:oracle}.

\subsection{A differentiation rule with controlled weight norm}
Choose an available Taylor order $r\le p$, and the smallest odd $n\ge r$. At Chebyshev nodes $t_j=\cos(j\pi/n)$, let $c_0=c_n=2$ and $c_j=1$ otherwise. For a real query radius $h>0$, use
\begin{equation}\label{eq:main-weights}
 d_j=\frac{(-1)^{(n-1)/2+j}}{n h c_jt_j^2},\qquad
 \sum_{j=0}^n d_jP(ht_j)=P'(0),\qquad
 W:=\sum_{j=0}^n|d_j|=n/h,
\end{equation}
where $P$ has degree at most $n$. Odd $n$ ensures no denominator vanishes. Appendix~\ref{app:differencing} proves the identity by differentiating the Lagrange cardinal polynomials.

For $\norm v\le V:=\max\{1,L_0/\mu\}$, the difference
\[
 \mathcal Z_{r,h}(x,y;v)=\sum_jd_j\nabla g(x,y+ht_jv)
\]
approximates $(g_{xy}v,g_{yy}v)$. The envelope gives a degree-$r$ Taylor remainder bounded by $A_r|t|^{r+1}$, with
\[
 A_r=\bar M(\bar\Lambda V)^{r+1}((r+1)!)^{\beta-1},
\]
where $\bar M,\bar\Lambda$ depend only on the fixed budgets. Choosing $h=(a/(nA_r))^{1/r}$ gives bias at most $a$ and
\begin{equation}\label{eq:main-W}
 W\le C_\beta r^\beta a^{-1/r}.
\end{equation}
The Taylor factorial is essential to this order-uniform estimate.

Allocate $b_j=\lceil B|d_j|/W\rceil$ independent gradient samples at node $j$, and sum their means with weights $d_j$. Then
\begin{equation}\label{eq:main-node-cost}
 \E\norm{\widehat Z_B-\mathcal Z_{r,h}}^2\le\sigma^2W^2/B,
 \qquad \sum_j b_j\le B+n+1.
\end{equation}
The node means are summed deterministically; randomizing one node could add variance from unbounded gradient means. No sample is reused. Real lower queries are unrestricted, so $h$ may grow with the order when $\beta<1$.

\subsection{Inner solution, adjoint, and outer step}
At each outer point, a restarted batched SGD solve finds $Y$ with $\E\norm{Y-y^*(x)}=O(\eps)$. At $(x,Y)$, the adjoint $v^*=g_{yy}^{-1}f_y$ lies in the ball of radius $V$. With fresh samples at each step, update
\begin{equation}\label{eq:main-adjoint}
 v_{k+1}=\Pi_V\!\left(v_k-\alpha_v
       [ (\widehat Z_{B_k}(x,Y;v_k))_y-\overline G_{f,y,k} ]\right),
 \qquad \alpha_v=1/(2L_1).
\end{equation}
The mean update has an additive error at most $a=O(\eps)$ and noise variance at most $\sigma^2(1+W^2)/B_k$. Projection keeps every queried direction inside the region of the uniform bias bound. A contraction estimate gives an adjoint error $O(\eps)$.

Both stochastic solves use increasing batches. For a recursion $E_{k+1}\le qE_k+d+A/B_k$, geometric allocation proportional to $q^{(J-1-k)/2}$ yields
\begin{equation}\label{eq:main-batches}
 E_J\le u^2,\qquad
 \sum_{k<J}B_k\le J+\frac{4A}{u^2(1-\sqrt q)^2},
\end{equation}
provided the initial and bias terms fit the target. Thus the logarithmic contraction length contributes additively, not multiplicatively. The adjoint loop's extra $n+1$ node-rounding samples per step are also explicitly counted.

Finally form, using new independent batches,
\[
 \widehat q=\overline G_{f,x}(x,Y)
              -(\widehat Z_{B_o}(x,Y;\widetilde v))_x.
\]
Hypergradient stability and the three errors give
$\E[\norm{\widehat q-\nabla\Phi(x)}\mid\text{outer history}]\le\eps/8$.
Normalized outer steps of length $\eta=\eps/(2L_\Phi)$ bound the distance from $x_0$ deterministically, so restarts at $y_0$ have no hidden tracking assumption. With $K=O(\eps^{-2})$ steps and an independent uniform output index, the descent inequality gives expected gradient norm at most $3\eps/4$.

\begin{algorithm}[t]
\caption{Fresh-sample Chebyshev--adjoint method}\label{alg:uniform}
\begin{algorithmic}[1]
\Require Accuracy $\eps$, order budget $p$, envelope and baseline constants, $(x_0,y_0)$.
\State Choose finite $r\le p$, odd $n$, $h,d_j,W$, and all loop lengths and batch schedules as in Appendices~\ref{app:differencing}--\ref{app:upper}.
\For{$t=0,\ldots,K-1$}
  \State Restart SGD on $g(x_t,\cdot)$ at $y_0$, using the geometric batches; obtain $Y_t$.
  \State Initialize $v_0=0$.
  \For{$k=0,\ldots,J_v-1$}
    \State At each $(x_t,Y_t+ht_jv_k)$ take $\lceil B_k|d_j|/W\rceil$ fresh $g$ samples.
    \State Form $\widehat Z_{B_k}$; independently average $B_k$ fresh $f$ gradients at $(x_t,Y_t)$.
    \State Apply the projected update~\eqref{eq:main-adjoint}.
  \EndFor
  \State Form $\widehat q_t$ from new independent $f$ samples and a new $g$-difference at $v_{J_v}$.
  \State $x_{t+1}\gets x_t-\eta\,\dir(\widehat q_t)$, where $\dir(q)=q/\norm q$ for $q\ne0$ and $\dir(0)=0$.
\EndFor
\State \Return $x_I$ for an independent uniform $I\in\{0,\ldots,K-1\}$.
\end{algorithmic}
\end{algorithm}

\subsection{Choosing the order without hiding its cost}
For $\beta>0$, minimize $r^\beta\eps^{-1/r}$ over integers $r\le p$. A minimizing order is $O_\beta(\log(1/\eps))$. For $\beta=0$, use $r=\min\{p,\lceil\log(1/\eps)\rceil\}$, with the natural interpretation at $p=\infty$; this is within a factor $e$ of the infimum. Equation~\eqref{eq:main-W} gives $W\le C\D_{p,\beta}(\eps)$. The complete fixed budget is bounded by
\begin{equation}\label{eq:main-total}
 C\eps^{-4}(1+W^2)+C\eps^{-2}\log^2(1/\eps)
 \le C\eps^{-4}\D_{p,\beta}(\eps)^2.
\end{equation}
The additive second term counts stencil ceilings, contractions, and restarts. No infinite stencil is used, even for $p=\infty$.

\section{A Globally Valid Lower Bound Uniform in Order}\label{sec:lower-overview}
The derivative envelope dictates how narrow a globally hidden signal may be. Appendices~\ref{sec:split}--\ref{sec:random} give complete proofs of the construction below.

\subsection{Separating descent from information}
A chain has range $O(T)$, so localizing its entire value risks violating the derivative budgets. A quadratic lift has an additional unbounded signal $-m(y-H(x))\nabla H(x)$ in its $x$-gradient. We therefore construct $H_T=A_T+B_T$ on $\R^T$ with
\begin{align}\label{eq:main-chain-bounds}
 |A_T|&=O(T),&\norm{\nabla A_T}&=O(\sqrt T),&\norm{\nabla^2 A_T}&=O(1),\notag\\
 |B_T|&=O(1),&\norm{\nabla B_T}&=O(1),&\norm{\nabla^2 B_T}&=O(1).
\end{align}
If $k$ is the first coordinate with $|z_k|<1$, then $\partial_k H_T(z)\le-1$. When all coordinates after $k$ lie in a fixed small interval, the entire gradient of $A_T$ depends only on the first $k$ coordinates, whereas the value and gradient of $B_T$ need at most the first $k+1$. Thus the extensive term supports descent without an exact upper gradient exposing the next direction.

The bounded part has the prefix form
\[
 B_T(z)=-\sum_{i=1}^T\left(\prod_{j<i}a(z_j)\right)\ell(z_i),
\]
with the support of $\ell$ inside a flat-zero region of the gate $a$. At most one summand contributes to the value or derivatives at a point. The inequality $|a'|^2\le Ca(1-a)$ controls long product derivatives through the probability of exactly one failure in independent Bernoulli trials. 

Let $U$ be an orthonormal $T$-frame and $\rho(w)=w/\sqrt{1+\norm w^2/\mathcal R^2}$, where $\mathcal R=\Theta(\sqrt T)$. Define
\[
 \mathcal A_U(w)=A_T(U^\top\rho(w))+\tfrac{\eta_0}{2}\norm w^2,
 \qquad \mathcal B_U(w)=B_T(U^\top\rho(w)).
\]
The bounded map controls arbitrary upper-query norms, and the quadratic prevents remote stationary escape. The wrapped sum has gradient norm at least $1/2$ whenever a latent coordinate has magnitude below one. Its initial gap is $O(T)$, and the bounds needed from~\eqref{eq:main-chain-bounds} remain uniform.

\subsection{A sine lift that respects all required derivatives}
Choose a sufficiently large fixed $K_0$, put $\tau=K_0^{-3/2}$, and for a small scale $s$ set
\begin{equation}\label{eq:main-lower-scale}
 h=\frac{K_0}{\D_{p,\beta}(s)},\qquad
 \lambda=\frac{s}{\tau},\qquad A=\frac{s^2}{\tau},\qquad T=\Theta(s^{-2}).
\end{equation}
Here $A$ is a scalar amplitude, distinct from $A_T$; the $h$ in this construction is unrelated to the algorithm's query radius. With fixed $0<b\le L_0$ and $\mu<m<L_1$, use the scalar lower variable
\begin{equation}\label{eq:main-lift}
 f_U(x,y)=A\mathcal A_U(x/\lambda)+by,\qquad
 g_U(x,y)=\frac m2y^2-m\frac Abh\,\mathcal B_U(x/\lambda)\sin(y/h).
\end{equation}
All derivatives of sine and cosine are bounded on the real line. The width satisfies
\begin{equation}\label{eq:main-width-bound}
 K_0s\le h\le K_0,\qquad
 s h^{-k}\le(k^\beta/K_0)^k\quad(1\le k\le p).
\end{equation}
Combined with $k!\ge(k/e)^k$, this bounds \emph{every} prescribed lower-variable derivative by its envelope using a single $K_0$. Baseline Hessian-block estimates are also uniform, including $p=1$ and the nonshrinking width at $\beta=0,p=\infty$.

The sine lift has a nonlinear lower root. Put $c=A/(bh)$ and define $\theta_c(B)$ by
\[
 \theta_c(B)=B\cos(c\theta_c(B)).
\]
For small $s$, this is a contraction and $\theta_c'(B)=1+O(c^2)$ uniformly on the bounded frontier range. The \emph{actual} minimizer and hyper-objective are
\[
 y_U^*(x)=\frac Ab\theta_c(\mathcal B_U(x/\lambda)),\qquad
 \Phi_U(x)=A[\mathcal A_U(x/\lambda)+\theta_c(\mathcal B_U(x/\lambda))].
\]
Consequently an unresolved coordinate still forces $\norm{\nabla\Phi_U(x)}\ge s/4$. The same accuracy threshold works for every $p$. The initialization has exact lower solution zero and gap at most $\Delta$.

\subsection{A fixed oracle and an adaptive information bound}
Write $\nabla g_U=(0,my)+v_U$. The full hidden signal, including the scalar lower-gradient component, satisfies
\[
 \sup_{x,y}\norm{v_U(x,y)}\le C_*sh.
\]
Choose $q_s=C_*^2s^2h^2/\sigma^2$ and return
\begin{equation}\label{eq:main-rare}
 G_g(x,y;Z)=(0,my)+\frac Z{q_s}v_U(x,y),\qquad Z\sim\Bern(q_s),
\end{equation}
independently at every lower call; return exact upper gradients. This samples gradients of two fixed smooth losses and is globally unbiased with variance at most $\sigma^2$.

Draw $U$ uniformly once. A simulator exposes at most one additional column on each successful Bernoulli round. Its responses are gradients of truncated \emph{composite} functions: artificial zero coordinates are constants, so their gradient components are projected out. Conditional on the simulator's prefix and past, unseen columns remain uniform on an orthogonal complement. A spherical-cap bound and the bounded image of $\rho$ control all $N$ adaptive queries and the arbitrary final output. At dimension $O(T\log(NT))$, transcript coupling fails with probability at most $1/8$. This calculation is performed in the simulator's unconditional prefix filtration, not by incorrectly conditioning on earlier good events.

If $Nq_s\le T/8$, the chance of exposing all $T$ columns is also at most $1/8$. Hence, with probability at least $3/4$, the output has gradient norm at least $s/4$. Taking $s=16\eps$ violates~\eqref{eq:target}, while
\begin{equation}\label{eq:main-lower-count}
 \frac{T}{q_s}=\Theta(s^{-4}h^{-2})
               =\Theta(\eps^{-4}\D_{p,\beta}(\eps)^2).
\end{equation}
Uniformity in the final comparison follows from $a^{-1}\D_{p,\beta}(\eps)\le\D_{p,\beta}(a\eps)\le\D_{p,\beta}(\eps)$ for fixed $a\ge1$. Averaging over the frame produces one fixed admissible instance against any given randomized algorithm. 

\section{Numerical Illustrations}\label{sec:experiments}
We examine two different consequences of the theory. The first experiment instantiates the lower-bound construction and measures how several methods remain above its accuracy barrier. The second holds a finite collection of problems fixed and measures the growth of the sample cost as the target accuracy decreases. These experiments use $p\in\{2,4\}$ and independent gradient samples. They illustrate the information obstruction and finite-range scaling, respectively; the minimax assertions are established by the proofs.

\subsection{Common oracle accounting and output criterion}
One complete gradient sample of either $f$ or $g$ costs one query; requesting both costs two. Every inner update, difference node, and final hypergradient estimate is charged. For each noise trajectory, we evaluate the mean of $\norm{\nabla\Phi(x_t)}$ over the pre-update outer iterates. This is the conditional expected error of the independent uniform output in Algorithm~\ref{alg:uniform}. We then average over eight independent trajectories. Reference lower solutions and hypergradients are used only by the evaluator and are never supplied to the algorithms. Appendix~\ref{app:experiments} gives the constructions and schedules, distinguishing practical settings from theorem-certified parameters.

\subsection{Stalling on the lower-bound construction}
For each allowed budget $N\in\{10^8,3\cdot10^8,\ldots,3\cdot10^{15},10^{16}\}$, we construct the sine-lift instance in~\eqref{eq:main-lift}. All methods receive the same instance at that budget. With fixed public regularity budgets, we maximize the admissible scale $s_N$ within the specified construction subject to $8Nq_{s_N}\le T_N$. Set $\eps_N=s_N/16$. The lower-bound argument gives
\[
 \E\norm{\nabla\Phi(\widehat x_N)}\ge 3s_N/16=3\eps_N,
 \qquad \norm{\nabla\Phi(x_0)}=s_N=16\eps_N.
\]
We compare Algorithm~\ref{alg:uniform} (Ours), F$^2$SA-$p$, its first-order specialization F$^2$SA-1~\citep{chen2026}, and the standard-gradient method of \citet{kwon2024}.

\begin{figure}[t]
 \centering
 \includegraphics[width=\textwidth]{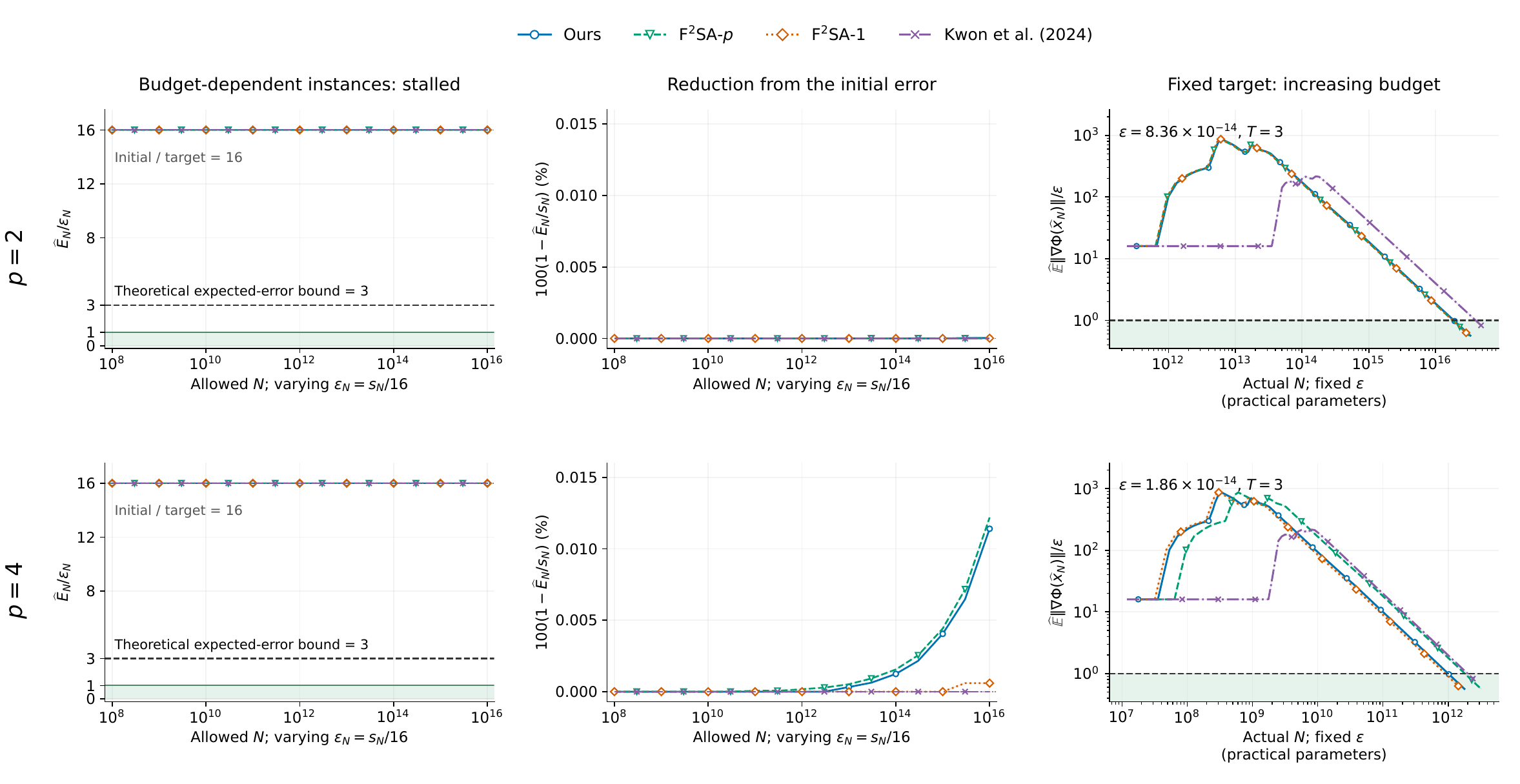}
 \caption{Lower-bound construction, with $p=2$ and $p=4$ in the two rows. Left: estimated uniform-output error $\widehat E_N$ divided by the budget-dependent target $\eps_N$; the dashed line at three is the theoretical expected-error lower bound. Middle: percentage reduction from the initial error $s_N$, exposing the small amount of optimization hidden by a changing instance scale. Right: a separate fixed three-coordinate chain for each $p$, with its target held fixed while actual sample cost increases. These last panels use practical schedules, including weighted node allocation for Ours. Shading shows $\pm1.96$ standard errors across eight noise trajectories. The three columns are different experimental views, not successive stages of one run.}\label{fig:lower-experiment}
\end{figure}

The first two columns of Figure~\ref{fig:lower-experiment} show that the tested methods remain close to their initial error and well above the target. The relative reduction is below $0.016\%$ throughout the displayed window. In particular, a decrease of the unnormalized error as $N$ increases would mainly reflect the scale $s_N$ of the newly constructed instance, rather than progress on a fixed objective. The third column separates this effect: on a fixed short chain, increasing the sample budget eventually brings all four methods below the same target within each panel. The early increase in gradient norm is retained. A lower bound for budgets below the revelation threshold does not imply permanent failure on a fixed finite chain.

\subsection{Cost growth on a fixed collection}
For the rate experiment, let $\mathcal H=\{H_4,H_8,H_{16}\}$ be the fixed collection defined in Appendix~\ref{app:experiment-family}. Each member has $d_x=d_y=4$, a nonlinear strongly convex lower problem, and an additional upper-coordinate term $\arctan(x_1/L)+\pi/2$ with $L\in\{4,8,16\}$. Its derivative is $[L(1+(x_1/L)^2)]^{-1}$, supplying a direction with slow progress while the other coordinates retain bilevel coupling. Coefficients, dimensions, initialization, and noise variance are held fixed across all accuracies and algorithms. It is a separate collection from the chain used for the lower bound.

Let $N_A(H,\eps)$ denote the first recorded sample budget at which the estimated uniform-output error plus $2.365$ standard errors is at most $\eps$. Define the observed collection maximum
\begin{equation}\label{eq:experiment-max}
 M_A(\eps)=\max_{H\in\mathcal H}N_A(H,\eps).
\end{equation}
The maximum is over separately solved problems; it is not an optimization of a pointwise maximum of their objectives. The standard-error margin is a diagnostic stopping rule, not a sequential confidence certificate. We use seven targets $\eps_j=0.12(0.8)^j$, $j=0,\ldots,6$, and repeat the three levels $j=0,3,6$ with independent noise. All 150 configurations, totaling 1200 trajectories, complete successfully. For each method, parameter scalings are frozen before this sweep, with practical constants and the source update recurrences.

\begin{figure}[t]
 \centering
 \includegraphics[width=\textwidth]{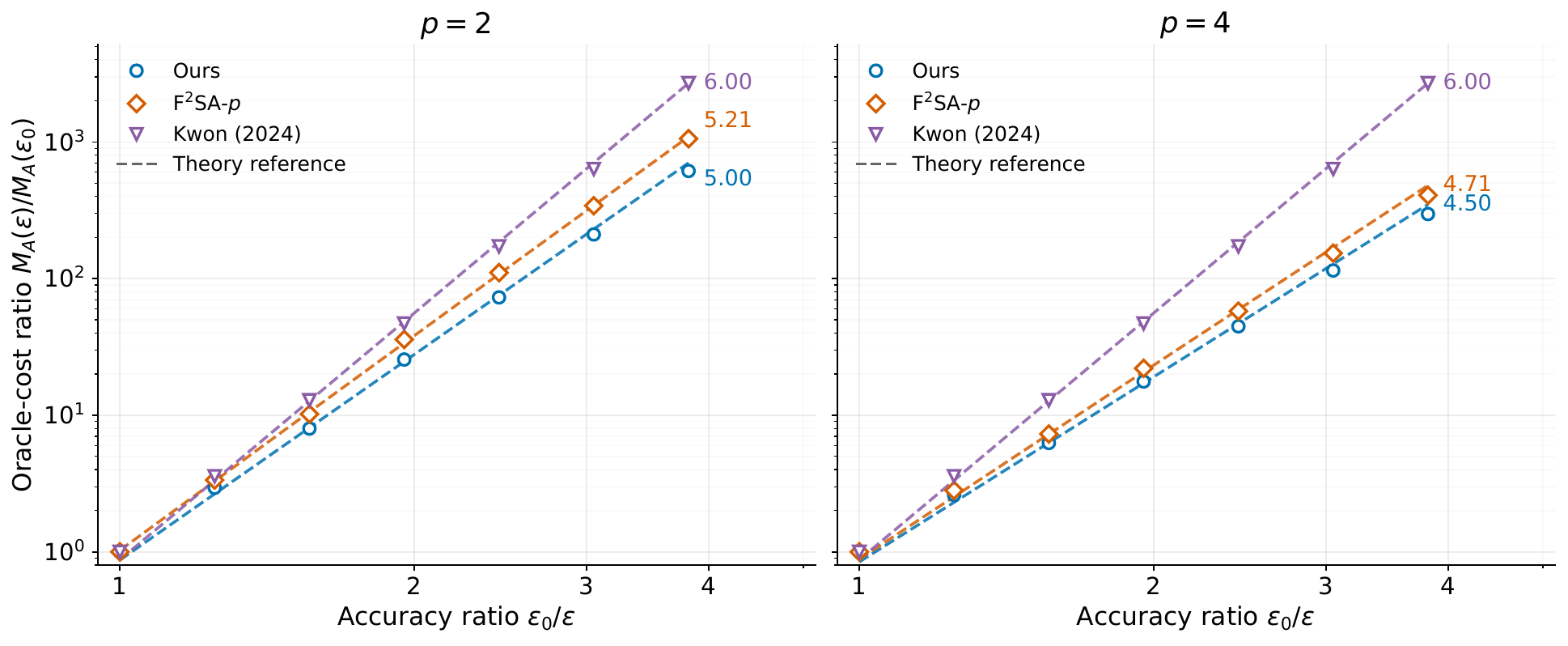}
 \caption{Cost growth on the fixed collection $\mathcal H$, using practical parameters. Markers show $M_A(\eps)/M_A(\eps_0)$ against $x=\eps_0/\eps$, where $\eps_0=0.12$, so the measured series start at $(1,1)$. Dashed curves have fixed theoretical shapes $S(x)=x^{4+2/p}$ (Ours), $S(x)=x^{4+2/p}\log(8x/\eps_0)/\log(8/\eps_0)$ (F$^2$SA-$p$), and $S(x)=x^6$ (Kwon). Only their heights are fitted: $\log C=\operatorname{mean}_j[\log y_j-\log S(x_j)]$. Labels give $[\log S(x_{\max})-\log S(1)]/\log x_{\max}$; empirical fits are not drawn. The factor eight inside the logarithm is a fixed reference convention, not an identified theorem constant. The guides are scaling references, not certified numerical upper bounds.}\label{fig:upper-experiment}
\end{figure}

Figure~\ref{fig:upper-experiment} compares growth after removing each method's initial cost. The main theoretical exponent is $4+2/p$ for Ours and F$^2$SA-$p$, whereas the standard independent-gradient guarantee for Kwon has exponent six. The F$^2$SA-$p$ upper bound also contains a logarithm; we retain it in the reference shape. Across the seven levels, the empirical log--log slopes are $4.80$ and $5.21$ for Ours and F$^2$SA-$p$ at $p=2$, and $4.26$ and $4.50$ at $p=4$; the Kwon slope is $5.87$. Independent-noise repetitions give the same qualitative ordering. The outer-iteration slopes are about $2.08$ for the two normalized methods and $1.94$ for Kwon, while the remaining sample-cost growth comes from the prescribed sampling schedules.

The observed growth reflects both outer progress and accuracy-dependent sampling. Batch rounding and lower-order cost terms also affect this finite window. In absolute cost, F$^2$SA-$p$ uses fewer absolute samples than Ours at every measured target in this experiment; the ratio plot compares growth relative to each method's own starting cost. Appendix~\ref{app:experiments} reports the absolute costs. At fixed $p$, removing a multiplicative $p$ factor cannot change the precision exponent. The figure assesses accuracy scaling at two fixed orders. A separate order sweep with a shared derivative envelope would be needed to assess the uniform-in-$p$ gain experimentally.

\section{Discussion and Conclusion}\label{sec:discussion}
The matching law separates fixed-order exponents from uniform smoothness gains. For $\beta>0$, the effective order saturates around $\log(1/\eps)/\beta$, leaving the unavoidable factor $\log^{2\beta}(1/\eps)$. For $\beta=0$, finitely many derivatives at order $O(\log(1/\eps))$ suffice for $\Theta(\eps^{-4})$. Qualitative analyticity does not imply either quantitative normalization: the baseline hard family is entire in $y$ but has accuracy-dependent higher-derivative growth.

The theorem fixes $M,\Lambda,\beta$ and all baseline budgets while allowing the hardness dimension to grow. It characterizes lower-variable derivative envelopes; joint analyticity, optimal condition-number dependence, and general infinite derivative sequences remain natural extensions. Other observation models and stronger output guarantees also lead to distinct complexity questions.

The numerical study complements this characterization in two ways: the budget-dependent construction exhibits the information barrier, and the fixed collection illustrates cost growth under explicitly specified practical schedules. The latter comparison distinguishes relative growth from absolute sample cost. Testing order-dependent gains on a larger fixed collection is a useful direction for further experiments.

At the degenerate boundary $L_1=\mu$, coupling is forced to vanish and the class with possibly noisy upper gradients has rate $\Theta(\eps^{-4})$; Appendix~\ref{app:boundary} proves this separately. Within the nondegenerate fresh-gradient model, the uniform theorem identifies both the finite-order exponent and the exact infinite-order overhead, with globally valid lower-bound responses and fully counted upper-bound samples.

\label{main-text-end}
\section*{Reproducibility statement}
The theoretical appendices supply the interpolation identity, sample allocations, scalar hard functions, global regularity estimates, and adaptive transcript coupling. Appendix~\ref{app:experiments} specifies the numerical instances, output statistic, practical schedules, and sample accounting. The accompanying experiment directories retain coefficient arrays, raw trajectory summaries, independent-noise repetitions, and plotting sources. Algebraic and numerical checks complement, but do not replace, the proofs. The manuscript compiles using standard pdf\LaTeX{} and Bib\TeX{} with the included vector figures, without external assets or executable dependencies.
\section*{Broader impact}
This work develops information-theoretic foundations for stochastic optimization. The experiments use synthetic objectives and contain no human-subject or proprietary data.
\bibliographystyle{plainnat}
\bibliography{references}

@inproceedings{kwon2023,
  author = {Kwon, Jeongyeol and Kwon, Dohyun and Wright, Stephen and Nowak, Robert D.},
  title = {A Fully First-Order Method for Stochastic Bilevel Optimization},
  booktitle = {Proceedings of the 40th International Conference on Machine Learning},
  series = {Proceedings of Machine Learning Research}, volume = {202},
  pages = {18083--18113}, year = {2023}, publisher = {PMLR},
  url = {https://proceedings.mlr.press/v202/kwon23c.html}
}

@inproceedings{kwon2024,
  author = {Kwon, Jeongyeol and Kwon, Dohyun and Lyu, Hanbaek},
  title = {On the Complexity of First-Order Methods in Stochastic Bilevel Optimization},
  booktitle = {Proceedings of the 41st International Conference on Machine Learning},
  series = {Proceedings of Machine Learning Research}, volume = {235},
  pages = {25784--25811}, year = {2024}, publisher = {PMLR},
  url = {https://proceedings.mlr.press/v235/kwon24b.html}
}

@article{chen2025,
  author = {Chen, Lesi and Ma, Yaohua and Zhang, Jingzhao},
  title = {Near-Optimal Nonconvex-Strongly-Convex Bilevel Optimization with Fully First-Order Oracles},
  journal = {Journal of Machine Learning Research}, volume = {26}, number = {109},
  pages = {1--56}, year = {2025},
  url = {https://www.jmlr.org/papers/v26/23-1104.html}
}

@inproceedings{chen2026,
  author = {Chen, Lesi and Li, Junru and Chayti, El Mahdi and Zhang, Jingzhao},
  title = {Faster Gradient Methods for Highly-Smooth Stochastic Bilevel Optimization},
  booktitle = {International Conference on Learning Representations}, year = {2026},
  note = {Full-text version arXiv:2509.02937v2, March 7, 2026},
  url = {https://arxiv.org/abs/2509.02937v2}
}

@misc{ji2026,
  author = {Ji, Kaiyi},
  title = {Lower Complexity Bounds for Nonconvex-Strongly-Convex Bilevel Optimization with First-Order Oracles},
  year = {2026}, howpublished = {arXiv:2511.19656v3},
  note = {ICML 2026; subsequently merged and strengthened in Chen, Ji, and Zhang (2026)},
  url = {https://arxiv.org/abs/2511.19656v3}
}

@misc{chenji2026,
  author = {Chen, Lesi and Ji, Kaiyi and Zhang, Jingzhao},
  title = {On the Condition Number Dependency in Bilevel Optimization},
  year = {2026}, howpublished = {arXiv:2511.22331v4},
  note = {Revised merged version, August 11, 2026},
  url = {https://arxiv.org/abs/2511.22331v4}
}

@misc{gu2026,
  author = {Gu, Zhihao and Wu, Qilong and Yang, Junchi},
  title = {{SGHA}: A Single-Loop Fully First-Order Algorithm for Nonconvex-Strongly-Convex Bilevel Optimization},
  year = {2026}, howpublished = {arXiv:2608.23211v1},
  note = {August 24, 2026},
  url = {https://arxiv.org/abs/2608.23211v1}
}

@inproceedings{chu2024,
  author = {Chu, Tianshu and Xu, Dachuan and Yao, Wei and Zhang, Jin},
  title = {{SPABA}: A Single-Loop and Probabilistic Stochastic Bilevel Algorithm Achieving Optimal Sample Complexity},
  booktitle = {Proceedings of the 41st International Conference on Machine Learning},
  series = {Proceedings of Machine Learning Research}, volume = {235},
  pages = {8848--8903}, year = {2024}, publisher = {PMLR},
  url = {https://proceedings.mlr.press/v235/chu24a.html}
}

@article{arjevani2023,
  author = {Arjevani, Yossi and Carmon, Yair and Duchi, John C. and Foster, Dylan J. and Srebro, Nathan and Woodworth, Blake},
  title = {Lower Bounds for Non-Convex Stochastic Optimization},
  journal = {Mathematical Programming}, volume = {199}, number = {1--2},
  pages = {165--214}, year = {2023}, doi = {10.1007/s10107-022-01822-7},
  url = {https://arxiv.org/abs/1912.02365}
}

@article{carmon2020,
  author = {Carmon, Yair and Duchi, John C. and Hinder, Oliver and Sidford, Aaron},
  title = {Lower Bounds for Finding Stationary Points {I}},
  journal = {Mathematical Programming}, volume = {184}, number = {1--2},
  pages = {71--120}, year = {2020}, doi = {10.1007/s10107-019-01406-y},
  url = {https://arxiv.org/abs/1710.11606}
}

@inproceedings{yang2023fde,
  author = {Yang, Yifan and Xiao, Peiyao and Ji, Kaiyi},
  title = {Achieving $\mathcal{O}(\epsilon^{-1.5})$ Complexity in {Hessian/Jacobian}-free Stochastic Bilevel Optimization},
  booktitle = {Advances in Neural Information Processing Systems},
  volume = {36}, year = {2023},
  url = {https://proceedings.neurips.cc/paper_files/paper/2023/hash/7c3a8d20ceadb7c519e9ac1bb77a15ff-Abstract-Conference.html}
}

@article{fornberg1988,
  author = {Fornberg, Bengt},
  title = {Generation of Finite Difference Formulas on Arbitrarily Spaced Grids},
  journal = {Mathematics of Computation}, volume = {51}, number = {184},
  pages = {699--706}, year = {1988},
  doi = {10.1090/S0025-5718-1988-0935077-0}
}

@misc{gu2026firstorderlower,
  title         = {An {$\Omega(\kappa_y^8\epsilon^{-6})$} Lower Bound for Stochastic {NC-SC} Bilevel Optimization with First-order Oracles},
  author        = {Gu, Zhihao and Wu, Qilong and Yang, Junchi},
  year          = {2026},
  eprint        = {2609.21905},
  archivePrefix = {arXiv},
  primaryClass  = {math.OC},
  url           = {https://arxiv.org/abs/2609.21905}
}

@misc{pan2026optimalbilevel,
  title         = {Optimal Stochastic Bilevel Optimization with First-Order Oracles},
  author        = {Pan, Linxuan and Yang, Junchi},
  year          = {2026},
  eprint        = {2610.01843},
  archivePrefix = {arXiv},
  primaryClass  = {math.OC},
  url           = {https://arxiv.org/abs/2610.01843}
}
\clearpage
\appendix
\section{Chebyshev gradient differences and independent sampling}\label{app:differencing}

\subsection{An exact interpolation identity}

\begin{lemma}[Odd-degree Chebyshev differentiation]\label{lem:cheb}
Let $n\ge1$ be odd, $t_j=\cos(j\pi/n)$ for $j=0,\ldots,n$, and let
$c_0=c_n=2$, $c_j=1$ otherwise. Put $s_n=(-1)^{(n-1)/2}$. For any radius $h>0$ define
\begin{equation}\label{eq:cheb-weights}
 d_j=\frac{s_n(-1)^j}{n h c_j t_j^2}.
\end{equation}
No $t_j$ is zero. For every vector-valued polynomial $P$ of degree at most $n$,
\begin{equation}\label{eq:cheb-identities}
 \sum_{j=0}^n d_jP(ht_j)=P'(0),\qquad
 \sum_{j=0}^n d_j=0,\qquad
 W:=\sum_{j=0}^n|d_j|=\frac nh.
\end{equation}
\end{lemma}
\begin{proof}
Let $U_{n-1}$ be the polynomial defined by
$U_{n-1}(\cos\theta)=\sin(n\theta)/\sin\theta$, and put
$w(t)=(t^2-1)U_{n-1}(t)$. Its roots are exactly the $t_j$.
For odd $n$, $w$ is even, $w(0)=-s_n$, and
$w'(t_j)=n c_j(-1)^j$. Differentiating the Lagrange cardinal polynomial
$w(t)/((t-t_j)w'(t_j))$ at zero gives \eqref{eq:cheb-weights} after rescaling by $h$. This proves polynomial exactness and the zero-sum identity. Finally, the polynomial $T_n$ defined by $T_n(\cos\theta)=\cos(n\theta)$ has
$T_n(t_j)=(-1)^j$ and $T_n'(0)=n s_n$.
All terms $d_j(-1)^j$ have sign $s_n$, so exactness applied to $T_n(t/h)$ yields $W=n/h$.
\end{proof}

\subsection{Bias under the stated finite list of derivatives}

Fix a radius $V\ge\max\{1,L_0/\mu\}$. Let $r\le p$ be a positive integer and choose $n$ to be the smallest odd integer at least $r$; thus $r\le n\le r+1$. At any $(x,y)$ and any $v$ with $\norm v\le V$, define the mathematical gradient difference
\begin{equation}\label{eq:Doperator}
 \mathcal Z_{r,h}(x,y;v)=\sum_{j=0}^n d_j\nabla g(x,y+ht_jv).
\end{equation}
Write $J_y\nabla g(x,y)[v]=(g_{xy}v,g_{yy}v)$. This notation denotes a population derivative, not an additional oracle.

Choose fixed $\overline M\ge\max\{1,M,\sqrt2L_2\}$ and
$\overline\Lambda\ge\max\{1,\Lambda\}$. The top Lipschitz bound in \eqref{eq:high}, or the two $L_2$ block bounds if $r=1$, gives the degree-$r$ Taylor remainder
\[
 \left\|\nabla g(x,y+tv)-\sum_{k=0}^r\frac{t^k}{k!}
 D_y^k\nabla g(x,y)[v^{\otimes k}]\right\|
 \le A_r|t|^{r+1},
\]
where a uniform valid choice is
\begin{equation}\label{eq:Ar}
 A_r=\overline M(\overline\Lambda V)^{r+1}((r+1)!)^{\beta-1}.
\end{equation}
The integral Taylor remainder proves this inequality also in the Lipschitz top-derivative formulation. When $r=1$, the norm of the full second lower derivative is bounded by $\sqrt2L_2$ almost everywhere; the same remainder follows by integration along the line.
Lemma~\ref{lem:cheb} therefore yields
\begin{equation}\label{eq:Zbias}
 \norm{\mathcal Z_{r,h}(x,y;v)-J_y\nabla g(x,y)[v]}
 \le n A_rh^r.
\end{equation}
For a desired bias $a>0$, take
\begin{equation}\label{eq:radius}
 h=\left(\frac{a}{nA_r}\right)^{1/r}.
\end{equation}
There is no need to restrict $h$ by inner strong convexity: these are queries at arbitrary real lower-variable points, not minimizers of negatively perturbed objectives.

\begin{lemma}[Uniform size of the differentiation weights]\label{lem:W}
For $0<a<1$, the choice \eqref{eq:radius} satisfies
\begin{equation}\label{eq:Wbound}
 W=\frac nh\le C_\beta r^\beta a^{-1/r},
\end{equation}
where the constant depends on the fixed parameters and $V$, not on $r$.
\end{lemma}
\begin{proof}
Direct substitution gives
\[
 W=n^{1+1/r}\overline M^{1/r}(\overline\Lambda V)^{1+1/r}
 ((r+1)!)^{(\beta-1)/r}a^{-1/r}.
\]
The elementary factorial bounds imply
$c r\le((r+1)!)^{1/r}\le C r$ for all $r\ge1$.
Also $r\le n\le r+1$ and $n^{1/r}\le2$.
Use the upper factorial bound if $\beta\ge1$ and the lower one if $0\le\beta<1$. This proves \eqref{eq:Wbound} in both cases.
\end{proof}

\subsection{Independent node batches remove a multiplicative order cost}

Given any $B\ge1$, allocate
\begin{equation}\label{eq:node-batch}
 b_j=\left\lceil B\frac{|d_j|}{W}\right\rceil
\end{equation}
fresh $g$-gradient samples to node $(x,y+ht_jv)$. Let $\overline G_j$ be that sample mean and define
\begin{equation}\label{eq:sampled-Z}
 \widehat Z_B(x,y;v)=\sum_{j=0}^n d_j\overline G_j.
\end{equation}
Its conditional mean is \eqref{eq:Doperator}. Independence of every sample gives
\begin{align}
 \E\norm{\widehat Z_B-\mathcal Z_{r,h}}^2
 &\le\sigma^2\sum_j\frac{d_j^2}{b_j}
 \le\frac{\sigma^2W^2}{B},\label{eq:Zvariance}\\
 \sum_j b_j&\le B+n+1.\label{eq:Zcost}
\end{align}
No deterministic gradient-size bound is used in \eqref{eq:Zvariance}; the deterministic node means are summed, not randomized. This matters when $0\le\beta<1$ and the radius in \eqref{eq:radius} can grow with the order. Each sample is evaluated only once. Neither common-random-number cancellation nor sample smoothness is assumed.

\section{Complete fixed-budget upper-bound proof}\label{app:upper}

\subsection{Hypergradient stability}

Let $\kappa=L_1/\mu$. Strong monotonicity gives
$\norm{y^*(x)-y^*(x')}\le\kappa\norm{x-x'}$.
Implicit differentiation gives
\[
 \nabla\Phi(x)=f_x-g_{xy}g_{yy}^{-1}f_y
 \quad\text{at }(x,y^*(x)).
\]
At an arbitrary $(x,y)$ put
\[
 v^*(x,y)=g_{yy}(x,y)^{-1}f_y(x,y),\qquad
 T(x,y)=f_x(x,y)-g_{xy}(x,y)v^*(x,y).
\]
The bounds $\norm{f_y}\le L_0$, $\norm{g_{xy}}\le L_1$,
$\norm{g_{yy}^{-1}}\le\mu^{-1}$ show that $\norm{v^*}\le L_0/\mu$ and
\begin{equation}\label{eq:Cy}
 \norm{T(x,y)-T(x,y^*(x))}\le C_y\norm{y-y^*(x)},
 \quad C_y=L_1+\frac{L_2L_0}{\mu}
       +\frac{L_1L_2L_0}{\mu^2}+\frac{L_1^2}{\mu}.
\end{equation}
To verify this, subtract the two three-factor products and use
$\norm{H^{-1}-\widetilde H^{-1}}\le\mu^{-2}\norm{H-\widetilde H}$.
The identical calculation in $(x,y)$ gives a Lipschitz constant
\begin{equation}\label{eq:Lphi-new}
 L_\Phi=\sqrt{1+\kappa^2}\,C_y
\end{equation}
for $\nabla\Phi$.

\subsection{A batch-allocation lemma}

\begin{lemma}[Geometric allocation in a contraction]\label{lem:batch}
Suppose a squared-error recursion satisfies
\[
 E_{k+1}\le qE_k+d+\frac{A}{B_k},\qquad 0<q<1,
\]
with $E_0\le D^2$. For a target $u>0$, assume $d/(1-q)\le u^2/4$.
Choose the deterministic length
\[
 J=\max\left\{1,\left\lceil
 \frac{\log(4(D+1)^2/u^2)}{-\log q}\right\rceil\right\},
\]
which has $q^JD^2\le u^2/4$, and let
\begin{equation}\label{eq:geobatches}
 B_k=\max\left\{1,\left\lceil
 \frac{4A}{u^2(1-\sqrt q)}q^{(J-1-k)/2}
 \right\rceil\right\},\qquad 0\le k<J.
\end{equation}
Then $E_J\le3u^2/4$ and
\[
 \sum_{k=0}^{J-1}B_k\le J+\frac{4A}{u^2(1-\sqrt q)^2}.
\]
\end{lemma}
\begin{proof}
Unroll the recursion. The initial term is at most $u^2/4$ and the bias term is at most $d/(1-q)\le u^2/4$.
For the noise term, \eqref{eq:geobatches} implies
\[
 \sum_{k=0}^{J-1}\frac{Aq^{J-1-k}}{B_k}
 \le\frac{u^2(1-\sqrt q)}4\sum_{k=0}^{J-1}q^{(J-1-k)/2}
 \le u^2/4.
\]
Summing the ceiling bounds proves the cost assertion. If $A=0$, take $B_k=1$ and omit the noise term.
\end{proof}

\subsection{The ordinary inner solution}

Run a predetermined $K$ normalized outer steps of length $\eta$. Their iterates obey
$\norm{x_t-x_0}\le K\eta$ for every outcome. Thus, when restarting an inner solve from the supplied $y_0$,
\[
 \norm{y_0-y^*(x_t)}\le R_0+\kappa K\eta=:D_y.
\]
For $g(x_t,\cdot)$, batched SGD with stepsize $\alpha_y=\mu/L_1^2$ satisfies
\[
 \E\norm{Y_{k+1}-y^*(x_t)}^2
 \le q_y\E\norm{Y_k-y^*(x_t)}^2
       +\alpha_y^2\sigma^2/B_k,
 \qquad q_y=1-\mu^2/L_1^2\in(0,1).
\]
This follows by expanding the squared distance, using strong monotonicity and the $L_1$-Lipschitz gradient bound. Lemma~\ref{lem:batch} produces $Y$ with
\begin{equation}\label{eq:ytarget}
 \E\norm{Y-y^*(x_t)}\le u_y:=\frac{\eps}{32C_y}
\end{equation}
using at most
$C[\eps^{-2}+\log(2+D_y/\eps)]$ samples. The loop length and every batch size are fixed before observing the samples.

\subsection{A projected stochastic adjoint solve}

Fix the resulting $(x,Y)$, conditionally on everything so far. Write
$H=g_{yy}(x,Y)$, $b=f_y(x,Y)$, and $v^*=H^{-1}b$.
These objects are used in analysis only. The algorithm never receives $H$.
Project the adjoint iterate onto the Euclidean ball of radius
$V=\max\{1,L_0/\mu\}$, start at $v_0=0$, and use
\begin{equation}\label{eq:adjoint-step}
 v_{k+1}=\Pi_V\left(v_k-\alpha_v
 \left[(\widehat Z_{B_k}(x,Y;v_k))_y-\overline F_{y,k}\right]\right),
 \qquad \alpha_v=\frac1{2L_1},
\end{equation}
where $\overline F_{y,k}$ is an independent mean of $B_k$ fresh $f$-gradients, restricted to its $y$ block. At most $2B_k+n+1$ samples are used in a step. By \eqref{eq:Zbias} and \eqref{eq:Zvariance}, the conditional mean of the bracket equals
$Hv_k-b+e(v_k)$, where $\norm{e(v_k)}\le a$, and its conditional variance is at most
$\sigma^2(1+W^2)/B_k$.

For completeness, nonexpansiveness of projection and $v^*\in\{\norm v\le V\}$ give
\begin{align*}
 \E[\norm{v_{k+1}-v^*}^2\mid v_k]
 &\le\norm{(I-\alpha_vH)(v_k-v^*)-\alpha_ve(v_k)}^2
      +\frac{\alpha_v^2\sigma^2(1+W^2)}{B_k}\\
 &\le q_v\norm{v_k-v^*}^2+\frac{3\alpha_v a^2}{\mu}
      +\frac{\alpha_v^2\sigma^2(1+W^2)}{B_k},
 \qquad q_v=1-\alpha_v\mu/2.
\end{align*}
Indeed $\norm{I-\alpha_vH}\le1-\alpha_v\mu$ and Young's inequality bounds the cross term by
$(\alpha_v\mu/2)\norm{v_k-v^*}^2+(2\alpha_v/\mu)a^2$.
The bias floor is at most $6a^2/\mu^2$.

Take
\begin{equation}\label{eq:accuracy-choices}
 u_v=\frac{\eps}{32L_1},\qquad
 a=c_a\eps,\quad
 0<c_a\le\min\left\{\frac1{32},\frac{\mu}{256L_1}\right\}.
\end{equation}
Then $6a^2/\mu^2\le u_v^2/4$. This assigns a fixed fraction of the adjoint error budget to differencing bias; geometric batching controls the remaining stochastic error.
Lemma~\ref{lem:batch}, with $D=L_0/\mu$ and
$A=\alpha_v^2\sigma^2(1+W^2)$, produces $\widetilde v$ satisfying
\[
 \E[\norm{\widetilde v-v^*(x,Y)}\mid x,Y]\le u_v
\]
in a fixed budget of at most
\begin{equation}\label{eq:adjoint-cost}
 C\left[(1+W^2)\eps^{-2}+(n+1)\log(2/\eps)\right]
\end{equation}
samples. The estimate does not require the biased mean map in \eqref{eq:adjoint-step} to be linear, conservative, or Lipschitz: its uniform additive bias bound suffices.

\subsection{The outer direction and full sample accounting}

Using fresh independent samples, form
\[
 \widehat q=\overline F_x(x,Y)-(\widehat Z_{B_o}(x,Y;\widetilde v))_x,
\]
with $B_o$ samples in the upper mean and node allocations \eqref{eq:node-batch} in the lower difference. Set
$B_o=\max\{1,\lceil1024\sigma^2(1+W^2)\eps^{-2}\rceil\}$.
The conditional noise contribution in expected norm is at most $\eps/32$.
Combining \eqref{eq:Cy}, \eqref{eq:ytarget}, \eqref{eq:accuracy-choices}, and the adjoint error gives
\begin{equation}\label{eq:outer-error}
 \E[\norm{\widehat q-\nabla\Phi(x)}\mid\text{outer history}]
 \le C_yu_y+L_1u_v+a+\eps/32\le\eps/8.
\end{equation}
The entire outer direction has sample cost at most
\begin{equation}\label{eq:per-outer-cost}
 C\left[(1+W^2)\eps^{-2}+(n+1)\log(2/\eps)
       +\log(2+D_y/\eps)\right]
\end{equation}
Set
\[
 \eta=\frac{\eps}{2L_\Phi},\qquad
 K=\left\lceil\frac{4\Delta}{\eta\eps}\right\rceil,
 \qquad
 x_{t+1}=x_t-\eta\frac{\widehat q_t}{\norm{\widehat q_t}},
\]
with no movement when $\widehat q_t=0$. For any $a,z$, with the same convention at zero,
$\ip{a}{z/\norm z}\ge\norm a-2\norm{z-a}$.
The descent lemma and \eqref{eq:outer-error}, followed by telescoping, give for an independently uniform $I\in\{0,\ldots,K-1\}$,
\[
 \E\norm{\nabla\Phi(x_I)}
 \le\frac{\Delta}{K\eta}+\frac\eps4+\frac{L_\Phi\eta}{2}
 \le\frac{3\eps}{4}.
\]
Here $K=O(\eps^{-2})$ and $D_y=O(\eps^{-1})$.

If $\beta>0$, choose an integer $r\le p$ minimizing
$r^\beta\eps^{-1/r}$. Such a minimizer exists and has $r=O_\beta(\log(1/\eps))$ for small $\eps$. If $\beta=0$, choose
$r=\min\{p,\lceil\log(1/\eps)\rceil\}$, interpreting the minimum in the evident way for $p=\infty$. In this latter case
$r^\beta\eps^{-1/r}\le e\D_{p,0}(\eps)$.
Use this $r$ in \eqref{eq:radius}, with the bias in \eqref{eq:accuracy-choices}.
Lemma~\ref{lem:W} implies $W\le C\D_{p,\beta}(\eps)$ and $n=O_\beta(\log(1/\eps))$.
Summing \eqref{eq:per-outer-cost} over all outer iterations yields
\[
 N\le C\eps^{-4}(1+W^2)+C\eps^{-2}\log^2(1/\eps)
 \le C\eps^{-4}\D_{p,\beta}(\eps)^2,
\]
because $\D_{p,\beta}(\eps)\ge1$. All finite moments used in telescoping exist: the outer iterates lie in a deterministic bounded ball, and the inner recursions bound their own second moments. This proves the upper half of Theorem~\ref{thm:uniform} with constants and a threshold uniform in $p$. In particular, no multiplicative $\log(1/\eps)$ or $p$ was hidden in the inner solves: across-node rounding contributes only $n+1$ samples per adjoint step, and the across-time noise costs form a geometric sum.

\section{A split chain with a bounded frontier potential}\label{sec:split}

This section is the structural core of the lower bound. A usual sum of $T$ chain links has range $O(T)$. Embedding that entire sum in a narrow lower-variable perturbation would destroy either strong convexity or the finite-order constants. We instead split the chain into an extensive part $A_T$ and a bounded part $B_T$.

\subsection{Fixed one-dimensional functions}
All scalar functions in this subsection are fixed once and for all, independently of $T$ and accuracy. A concrete smooth step is
\[
 s_0(t)=\begin{cases}e^{-1/t}&t>0,\\0&t\le0,\end{cases}
 \qquad S(t)=\frac{s_0(t)}{s_0(t)+s_0(1-t)}.
\]
It is $C^\infty$, nondecreasing, zero on $(-\infty,0]$, and one on $[1,\infty)$. Define
\begin{align*}
 \Psi(t)&=S(2t-1),& a(t)&=\Psi(t)+\Psi(-t),\\
 \ell(t)&=t\left[1-S\left(8|t|-1\right)\right],& \delta&=1/16.
\end{align*}
Thus $a=0$ on $[-1/2,1/2]$, $a=1$ on $(-\infty,-1]\cup[1,\infty)$, and $0\le a\le1$. Also $\ell(t)=t$ for $|t|\le1/8$ and $\ell=0$ for $|t|\ge1/4$. Although the display uses $|t|$, the functions are smooth at zero because the cutoffs are locally constant there.

Let
\[
 \zeta(t)=S\left(\frac{|t|-\delta}{\delta}\right)\left[1-S(|t|-1)\right].
\]
This is nonnegative, even, smooth, compactly supported, zero on $[-\delta,\delta]$, and equal to one for $1/8\le|t|\le1$. Set
\begin{equation}\label{eq:M-choice}
 B_0=\norm\ell_\infty,\quad
 M_{\rm ch}=2+\norm{\ell'}_\infty+B_0\norm{a'}_\infty,\qquad
 r(t)=M_{\rm ch}\int_{-\infty}^t\zeta(u)\,du.
\end{equation}
Then $r$ is bounded, nonnegative, nondecreasing, smooth, constant on $[-\delta,\delta]$, and $r'(t)=M_{\rm ch}$ on $1/8\le|t|\le1$. Its derivative is even.

\subsection{The two potentials}
For $z\in\R^T$, define
\begin{align}
 A_T(z)&=-r(z_1)+\sum_{i=2}^T\left[\Psi(-z_{i-1})r(-z_i)-\Psi(z_{i-1})r(z_i)\right],\label{eq:A}\\
 P_i(z)&=\prod_{j=1}^i a(z_j),\qquad P_0=1,\label{eq:prefix}\\
 B_T(z)&=-\sum_{i=1}^T P_{i-1}(z)\ell(z_i),\qquad H_T=A_T+B_T.\label{eq:B}
\end{align}

\begin{lemma}[Dimension-free bounds for the bounded part]\label{lem:B-bounds}
There are absolute constants $C_1,C_2$ such that, for every $T$ and $z$,
\begin{equation}\label{eq:B-bounds}
 |B_T(z)|\le B_0,\qquad \norm{\nabla B_T(z)}\le C_1,\qquad
 \norm{\nabla^2 B_T(z)}\le C_2.
\end{equation}
\end{lemma}
\begin{proof}
The support of $\ell$ and of each of its derivatives is contained in a region where $a$ is identically zero. Once a nonzero $\ell(z_i)$ occurs, every later prefix vanishes. Thus at most one summand in \eqref{eq:B} is nonzero. For the first two derivatives used below, a zero gate in this region has all derivatives zero. If a summand has a nonzero derivative contribution, its index lies in the support of a derivative of $\ell$; the gate at that index is locally zero and annihilates every later contribution. Thus at most one index contributes to the value, gradient, or Hessian at a given point. It remains to bound the first two derivatives of an arbitrary prefix product $P=\prod_{j=1}^m a(z_j)$ uniformly in $m$.

Put $M_a=\norm{a''}_\infty$. Nonnegativity of $a$ and $1-a$, applied to their quadratic Taylor upper bounds, gives
\[
 |a'(t)|^2\le2M_a\min\{a(t),1-a(t)\}\le4M_a a(t)(1-a(t)).
\]
Write $C_a=4M_a$. The elementary identity/inequality
\begin{equation}\label{eq:exact-one}
 \sum_{j=1}^m(1-a_j)\prod_{k\ne j}a_k\le1,\qquad a_j\in[0,1],
\end{equation}
is the probability of exactly one failure in independent Bernoulli trials with success probabilities $a_j$. It implies
\[
 \norm{\nabla P}^2
 \le C_a\sum_j a_j(1-a_j)\prod_{k\ne j}a_k^2\le C_a.
\]
For positive $a_j$, put $b_j=a'_j/a_j$. The off-diagonal part of $\nabla^2P$ is
$P[bb^\top-\diag(b_j^2)]$. Its operator norm is at most
\[
 2P\sum_j b_j^2\le 2C_a\sum_j(1-a_j)\prod_{k\ne j}a_k\le2C_a.
\]
The diagonal part has norm at most $M_a$. The bounds extend to zero $a_j$ by continuity. Multiplying $P$ by the one possibly active $\ell(z_i)$ gives, for example,
\[
 C_1=B_0\sqrt{C_a}+\norm{\ell'}_\infty,
\quad C_2=B_0(M_a+2C_a)+2\sqrt{C_a}\norm{\ell'}_\infty+\norm{\ell''}_\infty.
\]
This proves the lemma, including at boundaries of the support where all relevant cutoff derivatives vanish.
\end{proof}

\begin{lemma}[The descent part and the frontier gradient]\label{lem:chain}
There are absolute constants $C_A,C_G,C_L$, independent of $T$, with
\begin{align}
 |A_T(z)|&\le C_A T,& \norm{\nabla A_T(z)}&\le C_G\sqrt T,&
 \norm{\nabla^2 A_T(z)}&\le C_L.\label{eq:A-bounds}
\end{align}
Moreover:
\begin{enumerate}[label=(\roman*)]
\item If $|z_i|\le\delta$, then $\partial_i A_T(z)=0$, irrespective of the other coordinates.
\item If $k$ is the first index with $|z_k|<1$, then
\begin{equation}\label{eq:frontier-gradient}
 \partial_k H_T(z)\le-1.
\end{equation}
\item For $0\le k\le T$, if $|z_j|\le\delta$ for all $j>k$, the vector $\nabla A_T(z)$ is supported in the first $k$ coordinates and, on those coordinates, depends only on $z_1,\ldots,z_k$. Under the same condition, $B_T(z)$ and $\nabla B_T(z)$ depend only on the first $\min\{T,k+1\}$ coordinates and the gradient is supported there.
\end{enumerate}
\end{lemma}
\begin{proof}
Each term of $A_T$ involves two adjacent coordinates and fixed bounded scalar functions. Its Hessian is tridiagonal with uniformly bounded entries; the row-sum bound proves \eqref{eq:A-bounds}.
Because $r,r',\Psi,\Psi'$ are nonnegative, every partial derivative of $A_T$ is nonpositive. Its self-link contribution is
\[
 -\Psi(z_{i-1})r'(z_i)-\Psi(-z_{i-1})r'(-z_i)
\]
(with the first term $-r'(z_1)$ at $i=1$), and its next-link contribution is
\[
 -\Psi'(z_i)r(z_{i+1})-\Psi'(-z_i)r(-z_{i+1}).
\]
Both vanish when $|z_i|\le\delta$, proving (i).

For (ii), every preceding gate equals one. The preceding self-link therefore contributes exactly $-r'(z_k)$, since $r'$ is even. Splitting \eqref{eq:B} at $k$ gives
\[
 \partial_k B_T(z)=-\ell'(z_k)-a'(z_k)\mathcal S,
\]
where $\mathcal S$ is a suffix prefix-sum of the same type, and $|\mathcal S|\le B_0$ by the argument for Lemma~\ref{lem:B-bounds}.
When $|z_k|\le1/8$, $\ell'=1$ and $a'=0$, and all contributions from $A_T$ are nonpositive. When $1/8\le|z_k|<1$, use $r'=M_{\rm ch}$ and \eqref{eq:M-choice}. This proves \eqref{eq:frontier-gradient}.

For (iii), the case $k=T$ is immediate. Suppose $k<T$. $r$ is constant and $r'$ is zero on the tail interval, and both $\Psi$ and $\Psi'$ vanish there. The derivative in coordinate $k$ may involve $r(\pm z_{k+1})$, but these are the same fixed constant throughout that interval. For $B_T$, the first tail gate $a(z_{k+1})$ and all its derivatives vanish. All later summands and gradient components are therefore zero. The remaining terms use at most $z_{k+1}$.
\end{proof}

\section{Controlling arbitrary query norms}\label{sec:wrapper}

Enlarge a fixed constant $C_H\ge1$ so that
$\norm{\nabla H_T(z)}\le C_H\sqrt T$ for every $T,z$.
Let
\[
 \eta_0=1/16,\qquad \mathcal R=64C_H\sqrt T,\qquad
 \rho(w)=\frac{w}{\sqrt{1+\norm w^2/\mathcal R^2}}.
\]
The symbol $\mathcal R$ is unrelated to the initial inner-error budget $R_0$.
For an orthonormal frame $U=(u_1,\ldots,u_T)\in\R^{d\times T}$, set
\begin{align}
 \mathcal A_U(w)&=A_T(U^\top\rho(w))+\frac{\eta_0}{2}\norm w^2,\label{eq:Awrapper}\\
 \mathcal B_U(w)&=B_T(U^\top\rho(w)),\qquad
 \mathcal H_U=\mathcal A_U+\mathcal B_U.\label{eq:Bwrapper}
\end{align}

\begin{lemma}[Bounds and nonstationarity after wrapping]\label{lem:wrapper}
Uniformly in $T,d,U$,
\begin{align}
 |\mathcal B_U|&\le B_0,& \norm{\nabla\mathcal B_U}&\le C_B,&
 \norm{\nabla^2\mathcal B_U}&\le C_{BB},\label{eq:wrappedB}\\
 \norm{\nabla^2\mathcal A_U}&\le C_{AA}.\label{eq:wrappedA}
\end{align}
If some coordinate of $U^\top\rho(w)$ has absolute value below one, then
\begin{equation}\label{eq:wrapped-gradient}
 \norm{\nabla\mathcal H_U(w)}\ge1/2.
\end{equation}
Also $\mathcal H_U(0)-\inf\mathcal H_U\le C_\Delta T$ for a fixed $C_\Delta>0$.
\end{lemma}
\begin{proof}
Writing $\gamma=(1+\norm w^2/\mathcal R^2)^{-1/2}$ gives
\[
 D\rho(w)=\gamma I-\frac{\gamma^3}{\mathcal R^2}ww^\top,
 \qquad \norm{D\rho}\le1,\qquad \norm{D^2\rho}\le C/\mathcal R.
\]
The chain rule, Lemma~\ref{lem:B-bounds}, and \eqref{eq:A-bounds} prove \eqref{eq:wrappedB}--\eqref{eq:wrappedA}; the factor $\sqrt T$ in $\norm{\nabla A_T}$ is canceled by $\mathcal R$.

Let $v=U^\top\rho(w)$ and $q=U\nabla H_T(v)$, so $\norm q\le C_H\sqrt T$. Then
\begin{equation}\label{eq:wrapper-grad-formula}
 \nabla\mathcal H_U(w)=\gamma q+
 w\left(\eta_0-\frac{\gamma^3}{\mathcal R^2}\ip wq\right).
\end{equation}
If $\norm w\ge\mathcal R/2$, the triangle inequality applied before expanding $D\rho$ gives
\[
 \norm{\nabla\mathcal H_U(w)}\ge\eta_0\norm w-\norm q
 \ge C_H\sqrt T\ge1.
\]
Otherwise $\gamma\ge2/\sqrt5$. Choose the first $k$ with $|v_k|<1$. By Lemma~\ref{lem:chain}, $q\cdot u_k\le-1$, and
$|w\cdot u_k|=|v_k|/\gamma\le\sqrt5/2$. The absolute value of the parenthesized scalar in \eqref{eq:wrapper-grad-formula} is at most
\[
 \frac1{16}+\frac{(\mathcal R/2)C_H\sqrt T}{\mathcal R^2}
 =\frac9{128}.
\]
Hence the $u_k$ component is at most
$-2/\sqrt5+(\sqrt5/2)(9/128)<-1/2$.
Finally, $A_T$ is bounded by $C_AT$, $B_T$ is bounded, and the quadratic is nonnegative. The gap claim follows, after enlarging $C_\Delta$ to cover every $T\ge1$.
\end{proof}

The derivative-zero and prefix-dependence conclusions in Lemma~\ref{lem:chain}(iii) remain valid for oracle information after this transformation. The radial factors, $w$, and $\norm w$ are known from the query. For example, $D\rho(w)U\nabla A_T(v)$ is determined by the revealed frame prefix whenever the remaining $v_j$ lie in $[-\delta,\delta]$.

\section{The sine lift and uniform population admissibility}\label{app:embedding}

Appendices~\ref{sec:split}--\ref{sec:wrapper} establish the split-chain and bounded-radius construction. Its output, for every chain length $T$ and orthonormal frame $U$, is a pair of smooth functions $\mathcal A_U,\mathcal B_U$ with constants independent of $T,d,U$ such that
\begin{align}
 |\mathcal B_U|&\le B_0,&
 \norm{\nabla\mathcal B_U}&\le C_B,&
 \norm{\nabla^2\mathcal B_U}&\le C_{BB},\label{eq:toolboxB}\\
 \norm{\nabla^2\mathcal A_U}&\le C_{AA}.\label{eq:toolboxA}
\end{align}
Their sum $\mathcal H_U=\mathcal A_U+\mathcal B_U$ has a gradient norm at least $1/2$ whenever any coordinate of $U^\top\rho(w)$ has magnitude below one. Also $\mathcal H_U(0)-\inf\mathcal H_U\le C_\Delta T$.
The map $\rho$ has image in a ball of radius $\mathcal R=C_R\sqrt T$.
On the hidden-coordinate interval $[-\delta,\delta]$, where $\delta=1/16$, the gradient of $\mathcal A_U$ reveals no new coordinate, whereas the value and gradient of $\mathcal B_U$ reveal at most the next coordinate.

\subsection{Order-dependent width with order-independent constants}

Choose fixed
\[
 0<b\le\min\{L_0,1\}/2,\qquad m=(\mu+L_1)/2.
\]
A large fixed number $K_0\ge1$ will be chosen below, and set
\begin{equation}\label{eq:tau}
 \tau=K_0^{-3/2}.
\end{equation}
For a small accuracy scale $s\in(0,1)$ define
\begin{equation}\label{eq:width}
 h=\frac{K_0}{\D_{p,\beta}(s)}
   =K_0\sup_{\substack{k\in\N\setminus\{0\}\\k\le p}}
          \frac{s^{1/k}}{k^\beta}.
\end{equation}
This formula includes $\beta=0,p=\infty$, where $h=K_0$.
For every $p$,
\begin{equation}\label{eq:h-basic}
 K_0s\le h\le K_0,\qquad
 s h^{-k}\le\left(\frac{k^\beta}{K_0}\right)^k\quad(1\le k\le p).
\end{equation}
If $p\ge2$, there is also the uniform lower bound
\begin{equation}\label{eq:h-square}
 h\ge K_0s^{1/2}/2^\beta.
\end{equation}
Set
\begin{equation}\label{eq:new-scales}
 \lambda=s/\tau,\qquad A=s^2/\tau,\qquad
 T=\left\lfloor\frac{\Delta\tau}{2C_\Delta s^2}\right\rfloor.
\end{equation}
Enlarge $C_\Delta$ if necessary to bound a gap with any bounded frontier of magnitude at most $B_0$. For sufficiently small $s$, $T\ge1$.

Define, for scalar $y$,
\begin{equation}\label{eq:sine-instance}
 \begin{split}
 f_U(x,y)&=A\mathcal A_U(x/\lambda)+by,\\
 g_U(x,y)&=\frac m2y^2-m\frac Ab h\,
                  \mathcal B_U(x/\lambda)\sin(y/h).
 \end{split}
\end{equation}
Unlike a compactly supported cutoff, the lower-variable dependence in \eqref{eq:sine-instance} is entire. The nonlinear inner root need not equal a linear frontier exactly; its uniformly small effect on the gradient barrier is estimated next.

\subsection{The actual inner root and the actual hypergradient}

Put $c=A/(bh)$. By \eqref{eq:h-basic},
\begin{equation}\label{eq:c-small}
 0<c\le s/(\tau bK_0)\longrightarrow0
 \quad\text{uniformly in }p.
\end{equation}
For $|B|\le B_0$, define $\theta_c(B)$ by
\begin{equation}\label{eq:theta}
 \theta_c(B)=B\cos(c\theta_c(B)).
\end{equation}
When $cB_0<1$, the right side is a contraction in its scalar argument, so the root is unique. It satisfies $|\theta_c(B)|\le|B|$, and implicit differentiation gives
\[
 \theta_c'(B)=\frac{\cos(c\theta_c(B))}
                   {1+cB\sin(c\theta_c(B))}.
\]
For $cB_0\le1/4$, use $|1-\cos z|\le z^2/2$ and $|\sin z|\le|z|$ to obtain
\begin{equation}\label{eq:theta-error}
 \sup_{|B|\le B_0}|\theta_c'(B)-1|\le2c^2B_0^2.
\end{equation}
The population strong convexity verified below implies that the unique inner solution is
\begin{equation}\label{eq:new-root}
 y_U^*(x)=\frac Ab\theta_c(\mathcal B_U(x/\lambda)),\qquad
 \Phi_U(x)=A\left[\mathcal A_U(x/\lambda)
                  +\theta_c(\mathcal B_U(x/\lambda))\right].
\end{equation}
Indeed $g_y=0$ is exactly \eqref{eq:theta} after this substitution.
Equations \eqref{eq:theta-error} and \eqref{eq:toolboxB} show that
\begin{equation}\label{eq:new-barrier}
 \norm{\nabla\Phi_U(x)}\ge s/4
\end{equation}
whenever some coordinate of $U^\top\rho(x/\lambda)$ has magnitude below one, provided $s$ is sufficiently small uniformly in $p$.

\subsection{Checking every population budget}

The upper function is $b$-Lipschitz in $y$ and has Hessian norm at most
$\tau C_{AA}$. All positive-order lower derivatives of its full gradient vanish.
The perturbation blocks in the full Hessian of $g$ satisfy
\begin{align}
 \norm{g_{xx}}&\le(m/b)\tau h C_{BB},\label{eq:gxx}\\
 \norm{g_{xy}}&\le(m/b)s C_B,\label{eq:gxy}\\
 |g_{yy}-m|&\le(m/b)\frac{s^2}{\tau h}B_0.\label{eq:gyy}
\end{align}
The derivatives needed for the two joint $L_2$ conditions satisfy
\begin{align}
 \norm{D_xg_{xy}}&\le(m/b)\tau C_{BB},&
 \norm{D_yg_{xy}}&\le(m/b)(s/h)C_B,\label{eq:joint-new1}\\
 \norm{D_xg_{yy}}&\le(m/b)(s/h)C_B,&
 |D_yg_{yy}|&\le(m/b)(s^2/(\tau h^2))B_0.\label{eq:joint-new2}
\end{align}
By \eqref{eq:h-basic} and \eqref{eq:tau},
\[
 \tau h\le K_0^{-1/2},\quad \tau=K_0^{-3/2},\quad
 s/h\le K_0^{-1},\quad s^2/(\tau h^2)\le K_0^{-1/2}.
\]
Consequently a sufficiently large fixed $K_0$ makes $\tau C_{AA}\le L_1$, the bound \eqref{eq:gxx} an arbitrarily small fraction of $L_1-m$, and the sums of the two bounds in each of \eqref{eq:joint-new1}--\eqref{eq:joint-new2} at most $L_2$. Equations \eqref{eq:gxy}--\eqref{eq:gyy} then become arbitrarily small uniformly in $p$ as $s\downarrow0$, because
$s^2/(\tau h)\le s/(\tau K_0)$.
This proves global $\mu$-strong convexity and full-gradient $L_1$-Lipschitzness for small $s$. This argument also covers the nonshrinking width $h=K_0$ at $\beta=0,p=\infty$.

For $k=2,\ldots,p$, the derivative controlling
$\Lip_y(D_y^k\nabla g)$ is $D_y^{k+1}\nabla g$.
All derivatives of sine and cosine have absolute value at most one, so
\begin{equation}\label{eq:all-order-raw}
 \norm{D_y^{k+1}\nabla g_U}
 \le\frac mb\left(C_B+\frac{s}{\tau h}B_0\right)s h^{-k}.
\end{equation}
If $p\ge2$, \eqref{eq:h-square} implies
$s/(\tau h)\le2^\beta\sqrt s/(\tau K_0)\le1$ for a sufficiently small $s$, with a threshold independent of $p$.
Put $C_0=(m/b)(C_B+B_0)$. The factorial inequality $k!\ge(k/e)^k$ and \eqref{eq:h-basic} give
\begin{equation}\label{eq:all-order-envelope}
 \norm{D_y^{k+1}\nabla g_U}
 \le C_0\left(\frac{e^\beta}{K_0}\right)^k(k!)^\beta
 \le M\Lambda^{k+1}((k+1)!)^\beta=L_{k+1},
\end{equation}
after increasing the same fixed $K_0$ so that
\[
 \frac{e^\beta}{K_0\Lambda}\le
 \min\{1/2,M\Lambda/(2C_0)\}.
\]
This one choice proves \emph{all} required orders simultaneously, even when $p=\infty$. There are no separate accuracy thresholds indexed by the derivative order.
For $p=1$, no bound in \eqref{eq:all-order-envelope} is required and the preceding baseline checks already apply.

At $x=0$, $\mathcal B_U(0)=0$, so $y_U^*(0)=0$. The initial lower error is zero. Since $|\theta_c(B)|\le B_0$ and the extensive potential has range $O(T)$ after subtracting the nonnegative radial term,
\[
 \Phi_U(0)-\inf\Phi_U\le A C_\Delta T\le\Delta/2.
\]
Thus \eqref{eq:sine-instance} is admissible for a single accuracy threshold independent of $p$.

\section{Global oracle validity and unrestricted lower bound}\label{sec:random}

\subsection{Two fixed sample losses}

Let $b_*(y)=(0,my)$ and $v_U(x,y)=\nabla g_U(x,y)-b_*(y)$.
The complete full-gradient signal obeys
\begin{equation}\label{eq:new-signal}
 \norm{v_U(x,y)}\le\frac mb\left(C_Bsh+B_0A\right)
 \le C_*sh\quad\text{for every real }x,y,
\end{equation}
where $C_*=(m/b)(C_B+B_0/(\tau K_0))$ is fixed.
For sufficiently small $s$, let
\begin{equation}\label{eq:new-q}
 q_s=C_*^2s^2h^2/\sigma^2\in(0,1).
\end{equation}
Each $g$ call draws a fresh Bernoulli$(q_s)$ coin $Z$ and returns
\begin{equation}\label{eq:new-oracle}
 G_g(x,y;Z)=b_*(y)+\frac Z{q_s}v_U(x,y).
\end{equation}
Every $f$ call returns its exact gradient. Directly,
\[
 \E G_g=\nabla g_U,\qquad
 \E\norm{G_g-\nabla g_U}^2
 =(q_s^{-1}-1)\norm{v_U}^2\le\sigma^2.
\]
The oracle is even the gradient of the fixed random sample loss
\[
 g_{U,Z}(x,y)=\frac m2y^2-\frac Z{q_s}m\frac Ab h\,
                   \mathcal B_U(x/\lambda)\sin(y/h).
\]
The two losses and their probabilities are selected before the algorithm runs. Population regularity, not samplewise uniform smoothness or samplewise convexity, is the requirement in the model. In particular, the oracle is unbiased globally, not only near $y^*(x)$.

\subsection{A spherical-cap bound}

\begin{lemma}\label{lem:cap-new}
If $u$ is uniform on the unit sphere in an $m_0$-dimensional Euclidean space and $z$ is fixed with $\norm z\le R$, then
\[
 \Prob\{|\ip uz|>a\}\le2\exp\left(-\frac{m_0a^2}{2R^2}\right).
\]
\end{lemma}
\begin{proof}
To compute the spherical moments, write a standard Gaussian vector in polar form $Z=Ru$. Its rotationally invariant density factors into a radial density and the uniform surface measure, so $R$ and $u$ are independent. One-dimensional Gaussian integration by parts gives $\E Z_1^{2k}=(2k-1)!!$, while integrating the radial density $r^{m_0-1}e^{-r^2/2}$ gives $\E R^{2k}=m_0(m_0+2)\cdots(m_0+2k-2)$. Dividing and rotating proves
$\E[\ip{u}{e}^{2k}]=(2k-1)!!/[m_0(m_0+2)\cdots(m_0+2k-2)]$
for a unit vector $e$. Odd moments vanish by symmetry. Expanding the exponential therefore gives
$\E e^{t\ip uz}\le e^{t^2\norm z^2/(2m_0)}$.
Exponential Markov bounds for the positive and negative tails, optimized in $t$, prove the claim.
\end{proof}

\subsection{A deferred-prefix simulation}

Fix any adaptive randomized algorithm with a fixed budget at most $N$; include its seed in the probability space and pad early stopping with unused calls. Draw $U$ uniformly from the orthonormal $T$-frames in $\R^d$. Independently attach Bernoulli$(q_s)$ coins $Z_1,\ldots,Z_N$ to all rounds, including $f$ rounds. A coin on an $f$ round is a bookkeeping device only. Write
$k_t=\min\{T,\sum_{i\le t}Z_i\}$.

Construct a simulation whose transcript before round $t$ is measurable using only the seed, preceding coins, and $u_1,\ldots,u_{k_{t-1}}$. For an upper query, use the gradient of the composite extensive potential with all later latent coordinates set to the constant zero. The known radial term and the scalar $f_y=b$ are retained. For a failed lower query return $b_*(y)$. For a successful lower query, reveal the next column to the simulation and use the bounded frontier with later latent coordinates set to zero. Differentiate this \emph{truncated composite function}; artificial zero coordinates are constants, so no unrevealed gradient component is retained. On a successful auxiliary coin at an upper query, reveal one more column to the simulation after sending the upper response. The algorithm itself is never directly handed a frame column. Once $k_t=T$, use the full formulas.

Appendix~\ref{sec:split} proves the needed prefix properties. They show inductively that this simulation has the claimed measurability. Conditional on its current prefix, every unseen frame column is marginally uniform on the unit sphere in the orthogonal complement, of dimension at least $d-T$. Since $\norm{\rho(x/\lambda)}\le\mathcal R$, Lemma~\ref{lem:cap-new} and a union bound imply that the probability of any unseen coordinate exceeding $\delta$ at any query or the final output is at most
\begin{equation}\label{eq:new-bad-event}
 2(N+1)T\exp\left(-\frac{(d-T)\delta^2}{2\mathcal R^2}\right).
\end{equation}
This calculation is performed in the unconditionally prefix-measurable simulation; it does not condition a future random direction on a previous no-failure event.

On the complement of this bad event, the simulation and real transcripts agree inductively. An upper gradient depends only on the old prefix; a failed lower gradient is $b_*(y)$; and a successful lower gradient needs at most the old prefix plus one column. Sine and cosine depend on the chosen scalar $y$, not on the frame, so the conclusion holds for every size and sign of the lower query. The bounded-radius map handles every upper-query norm. The argument also includes any final output not previously queried.

Choose
\[
 d\ge T+\frac{2\mathcal R^2}{\delta^2}\log(16(N+1)T).
\]
Then \eqref{eq:new-bad-event} is at most $1/8$. If $Nq_s\le T/8$, Markov's inequality gives
$\Prob\{\sum_{t=1}^N Z_t\ge T\}\le1/8$.
With probability at least $3/4$, a column remains unseen at the output and its output coordinate has magnitude at most $\delta<1$. Equation~\eqref{eq:new-barrier} gives
\begin{equation}\label{eq:new-expected-lower}
 \E_{U,\mathrm{oracle},\mathrm{seed}}\norm{\nabla\Phi_U(\widehat x)}
 \ge3s/16.
\end{equation}
Averaging implies that at least one fixed frame has this lower bound on the algorithm's expected gradient norm. No zero-respecting restriction has been imposed.

\subsection{Uniform comparison with the target rate}

Take $s=16\eps$. The admissibility threshold established above is uniform in $p$, and \eqref{eq:new-expected-lower} becomes $3\eps$. Since $T\ge c s^{-2}$ for small $s$,
\[
 \frac{T}{8q_s}\ge c s^{-4}h^{-2}
       =c s^{-4}\D_{p,\beta}(s)^2.
\]
For any fixed $a\ge1$ with $a\eps<1$,
\[
 a^{-1}\D_{p,\beta}(\eps)
 \le\D_{p,\beta}(a\eps)\le\D_{p,\beta}(\eps),
\]
because $a^{-1}\le a^{-1/r}\le1$ term by term, also when the infimum is not attained. Therefore every sufficiently small fixed budget
$c\eps^{-4}\D_{p,\beta}(\eps)^2$ fails on an admissible fixed instance.
This proves the lower half of Theorem~\ref{thm:uniform}.
Finally, $\D_{p,\beta}(\eps)\le\eps^{-1}$ by choosing $r=1$, so the lower-bound budget is at most a fixed multiple of $\eps^{-6}$. Since $\mathcal R^2=O(T)$, the specified dimension is uniformly
$O(\eps^{-2}\log(1/\eps))$.

\section{Consequences and the quantifiers in the main results}\label{app:consequences}\label{sec:qualitative}

\begin{table}[ht]
\centering\small
\caption{Consequences under the same fresh-gradient contract and nondegenerate baseline budgets. Constants are uniform in order only in the envelope rows. The displayed rates retain all logarithmic factors.}\label{tab:regimes}
\begin{tabularx}{\textwidth}{@{}Y l@{}}
\toprule
Population regularity in the lower variable & Minimax fixed sample budget\\
\midrule
Baseline ($p=1$) & $\Theta(\eps^{-6})$\\
Arbitrary fixed finite list, order $p$ & $\Theta_p(\eps^{-4-2/p})$\\
Envelope $M\Lambda^j(j!)^\beta$, any $p$ & $\Theta(\eps^{-4}\D_{p,\beta}(\eps)^2)$\\
Infinite exponential envelope ($\beta=0$) & $\Theta(\eps^{-4})$\\
Infinite analytic envelope ($\beta=1$) & $\Theta(\eps^{-4}\log^2(1/\eps))$\\
Infinite factorial-power envelope, fixed $\beta>0$ & $\Theta_\beta(\eps^{-4}\log^{2\beta}(1/\eps))$\\
Qualitative $C^\infty$, even entire in $y$, without an envelope & $\Theta(\eps^{-6})$\\
\bottomrule
\end{tabularx}
\end{table}
\subsection{Arbitrary fixed finite budgets are preserved}
We give the class-inclusion proof of Corollary~\ref{thm:fixed}. Fix $p\ge2$ and positive $L_3,\ldots,L_{p+1}$. Put
\[
 M_-=\min_{3\le j\le p+1}L_j,\qquad
 M_+=\max_{3\le j\le p+1}L_j.
\]
Keep all baseline budgets, including $L_2$, unchanged, and let $\cF^-$ and $\cF^+$ be the classes in Definition~\ref{def:profile} with $\beta=0$, $\Lambda=1$, and envelope constants $M_-$ and $M_+$, respectively. The arbitrary finite-list class $\cF^{\rm list}$ satisfies
\[
 \cF^-\subseteq\cF^{\rm list}\subseteq\cF^+.
\]
For every $q$ in~\eqref{eq:high}, these are exactly the componentwise inequalities between the derivative budgets; the $q=1$ bound on the $f$ derivative still uses the common baseline $L_2$. Oracle and initialization contracts do not change. Hence their minimax complexities obey
\[
 Q^-(\eps)\le Q^{\rm list}(\eps)\le Q^+(\eps).
\]
Theorem~\ref{thm:uniform}, applied twice with $\beta=0$, gives the same exponent $4+2/p$ on both sides. The constants may depend on $M_-$ and $M_+$, and therefore on the fixed list and $p$, as claimed. For $p=1$, the additional list is empty and Theorem~\ref{thm:uniform} applies directly.

The lower instances have exact $f$ gradients for every choice of envelope. Thus the same inclusion and upper bound apply when every upper gradient is exact. This establishes the lower-only-noise version as well. No comparison across orders with different envelope constants is made in this argument.

\subsection{Optimizing the order}
Let $t=\log(1/\eps)$. For $\beta>0$, set $\phi(r)=\beta\log r+t/r$ for real $r\ge1$. Its derivative is
\[
 \phi'(r)=\frac{\beta r-t}{r^2}.
\]
For sufficiently small $\eps$, $r_*=t/\beta\ge1$. If the finite integer $p\le r_*$, the objective decreases throughout $[1,p]$, so its integer minimum equals $p^\beta e^{t/p}$. If $p\ge r_*$ (including $\infty$), the continuous minimum is
\[
 r_*^\beta e^{t/r_*}=e^\beta(t/\beta)^\beta.
\]
The feasible integer $\lceil r_*\rceil$ is at most twice $r_*$ and, since $p$ is an integer, is no greater than $p$ whenever $p\ge r_*$. Its value is at most $2^\beta$ times the continuous minimum. These comparisons prove~\eqref{eq:phase}, with constants depending only on the fixed parameters and $\beta$. They also show that a minimizing integer is $O_\beta(t)$.

When $\beta=0$, $e^{t/r}$ is decreasing, so $\D_{p,0}(\eps)=e^{t/p}$ for finite $p$ and $\D_{\infty,0}(\eps)=1$. If $p\ge\lceil t\rceil$, using $r=\lceil t\rceil$ gives $e^{t/r}\le e$; otherwise using $r=p$ attains the finite infimum. This proves both the endpoint and the finite-order algorithmic choice in Appendix~\ref{app:upper}. No infinite stencil is used.

\subsection{Qualitative smoothness does not give a uniform improvement}
\begin{corollary}[Smooth or entire instances without a shared derivative envelope]\label{cor:bare}
Keep only the nondegenerate baseline budgets and Assumption~\ref{ass:oracle}. Further require both population functions to be jointly $C^\infty$, or require in addition that they be analytic in $y$, but impose no common quantitative bounds on the high derivatives. The minimax complexity remains $\Theta(\eps^{-6})$. The hard functions can be entire in the scalar lower variable.
\end{corollary}
\begin{proof}
The upper algorithm for $p=1$ uses only the baseline assumptions. For the lower bound, use the sine lift of Appendix~\ref{app:embedding} with $p=1$, hence $h=K_0s$. The split-chain functions are jointly smooth and both population functions are entire in $y$. All higher derivatives exist for each instance, but their unprescribed bounds can increase as $s\downarrow0$. The lower budget is $\Theta(s^{-4}h^{-2})=\Theta(s^{-6})$. These instances meet the additional qualitative requirements and therefore prove the claim.
\end{proof}
The sine is entire regardless of its frequency, so a complex radius without an associated uniform norm or growth bound is also not enough to exclude this example. In contrast, the factorial envelope in~\eqref{eq:profile} bounds the derivative growth at every real lower-variable point, uniformly over $x$, and therefore does exclude the baseline high-frequency family at $p=\infty$.

\subsection{Which assumptions the uniform upper method uses}
The Chebyshev--adjoint upper proof uses high lower-variable derivatives of $g$, together with baseline smoothness and $y$-Lipschitzness of $f$; it does not need the extra positive-order bounds on $D_y^q\nabla f$. We retain those bounds in the main model to preserve the finite-order framework being compared. The hard upper function satisfies them identically. Thus their strengthening to joint Lipschitzness leaves the stated matching results unchanged. No analogous claim is made for uniform bounds on pure $x$ derivatives of $g$.

\section{Degenerate budgets and a self-contained single-level comparison}\label{app:boundary}
This appendix allows noisy upper gradients, as in Assumption~\ref{ass:oracle}; it does not assert the same lower bound when the upper oracle is made exact. The strict inequality $L_1>\mu$ in Theorem~\ref{thm:uniform} is necessary for the stated rate when numerical budgets are fixed literally. Here we prove the alternative boundary rate without invoking a lower-bound theorem as a black box.

\begin{proposition}[Forced decoupling]\label{prop:decoupling}
Under the baseline population assumptions with $L_1=\mu$, one has
\[
 g_{yy}(x,y)=\mu I,\qquad g_{xy}(x,y)=0,\qquad
 g_y(x,y)=\mu y+c
\]
for a vector $c$ independent of $x$ and $y$. In particular, $y^*(x)=-c/\mu$ is constant.
\end{proposition}
\begin{proof}
For any unit lower direction $v$, differentiating the full-gradient Lipschitz bound in that direction gives
\[
 \norm{(g_{xy}v,g_{yy}v)}\le\mu.
\]
Strong convexity gives $\ip{v}{g_{yy}v}\ge\mu$, hence $\norm{g_{yy}v}\ge\mu$. Equality must hold throughout. Consequently $g_{xy}v=0$ and $g_{yy}v=\mu v$ for every unit $v$. The blocks are continuous by assumption, so the identities hold everywhere. Integrating them gives the claimed affine lower gradient.
\end{proof}

\begin{proposition}[Boundary minimax rate]\label{prop:boundary-rate}
Fix positive baseline budgets except set $L_1=\mu$, and let $R_0\ge0$. Impose either a fixed finite positive derivative list or the envelope of Definition~\ref{def:profile}, including $p=\infty$. Under the same fresh-gradient oracle and dimension convention, the minimax complexity is $\Theta(\eps^{-4})$ for all sufficiently small $\eps$.
\end{proposition}
\begin{proof}[Upper bound]
Write $\bar y=-c/\mu$. At $(x_0,y_0)$, take $M$ fresh samples and estimate
\[
 \widehat y=y_0-\frac1{\mu M}\sum_{i=1}^M G_{g,y}(x_0,y_0;i).
\]
Then $\E\norm{\widehat y-\bar y}^2\le\sigma^2/(\mu^2M)$. Set
$d=\min\{1,\eps/(16L_1)\}$ and choose $M=\max\{1,\lceil\sigma^2/(\mu^2d^2)\rceil\}$. For the random but subsequently fixed objective $\widehat h(x)=f(x,\widehat y)$,
\begin{align*}
 \Lip(\nabla\widehat h)&\le L_1,\\
 \E[\widehat h(x_0)-\inf\widehat h]&\le\Delta+2L_0d=:D_h,\\
 \E\norm{\nabla\widehat h(x)-\nabla\Phi(x)}&\le L_1d.
\end{align*}
The last estimate holds also at a data-dependent $x$, because its pointwise bound $L_1\norm{\widehat y-\bar y}$ is uniform in $x$. The gap estimate follows from $|f(x,\widehat y)-f(x,\bar y)|\le L_0\norm{\widehat y-\bar y}$, uniformly in $x$.

Run normalized stochastic gradient descent on $\widehat h$ with fresh batches of size $S_h=\max\{1,\lceil256\sigma^2/\eps^2\rceil\}$, step length $\eta_h=\eps/(4L_1)$, and $K_h=\max\{1,\lceil8D_h/(\eta_h\eps)\rceil\}$ steps. Conditional expected gradient-estimation error is at most $\eps/16$. The descent inequality proved in Appendix~\ref{app:upper}, now averaged also over $\widehat y$, gives
\[
 \frac1{K_h}\sum_{k<K_h}\E\norm{\nabla\widehat h(x_k)}
 \le\frac{D_h}{K_h\eta_h}+\frac\eps8+\frac{L_1\eta_h}{2}
 \le\frac{3\eps}{8}.
\]
Adding the uniform gradient discrepancy of at most $\eps/16$ verifies the target for an independently uniform iterate. The fixed cost is $M+K_hS_h=O(\eps^{-4})$.
\end{proof}

\begin{proof}[Lower bound]
Use the split chain and wrapper of Appendices~\ref{sec:split}--\ref{sec:wrapper}. Choose a fixed $\tau>0$ sufficiently small that $\tau(C_{AA}+C_{BB})\le L_1$. For small $s$, put $\lambda=s/\tau$, $A=s^2/\tau$, and $T=\lfloor\Delta\tau/(2C_\Delta s^2)\rfloor$. Set
\[
 f_U(x,y)=A\mathcal H_U(x/\lambda),\qquad
 g_U(x,y)=\frac\mu2\norm y^2.
\]
The upper objective is independent of $y$, its gradient is $L_1$-Lipschitz, the gap is at most $\Delta$, and every additional lower-variable derivative vanishes. The lower initial error is zero. Give exact $g$ gradients.

For the $f$ oracle use the baseline $\overline f_U=A\mathcal A_U(x/\lambda)$ and perturbation $A\mathcal B_U(x/\lambda)$. The full gradient difference $v_U=\nabla f_U-\nabla\overline f_U$ has norm at most $sC_B$ globally. Let $q_s=(sC_B/\sigma)^2$ and return
\[
 G_f=\nabla\overline f_U+(Z/q_s)v_U,\qquad Z\sim\Bern(q_s),
\]
independently on each call. This is the gradient of a fixed two-point distribution of smooth losses and is globally unbiased with variance at most $\sigma^2$.

On small hidden coordinates, a failed $f$ response depends only on the already exposed frame prefix, by Lemma~\ref{lem:chain}(iii), while a successful response uses at most one additional column. Repeating the simulator in Appendix~\ref{sec:random} with the roles of the calls exchanged therefore gives the same coupling and arbitrary-output guarantee: for $Nq_s\le T/8$ and $d=O(T\log(NT))$, the expected gradient norm is at least $3s/8$. Here the failed response is the prefix-local baseline instead of a vector independent of $U$; prefix measurability is all the simulation requires. Taking $s=8\eps$ yields $N=\Omega(T/q_s)=\Omega(\eps^{-4})$.
\end{proof}

If $\Delta=0$, the supplied point is a global minimizer of a differentiable function and $\nabla\Phi(x_0)=0$, so zero samples suffice. If $L_0=0$, $f$ is independent of $y$ and the same single-level rate applies for positive gap and variance budgets. If $L_1<\mu$ and $d_y\ge1$, the class is empty. Other zero smoothness budgets may enforce further structural restrictions and are not included in Theorem~\ref{thm:uniform}.

\section{Detailed comparison and proof checks}\label{app:comparison}
The results compared below have different quantifiers and oracle capabilities. A common precision exponent is not sufficient to identify their minimax classes. The versions and primary-source URLs used here are specified in the bibliography.

\subsection{Accuracy, budgets, and regularity}
\paragraph{Accuracy conversion.}
If a theorem guarantees $\E\norm{\nabla\Phi}^2\le\varepsilon$, choosing $\varepsilon=\eps^2$ suffices for our target by Jensen's inequality. This is why the original $\widetilde O(\varepsilon^{-7/2})$ F$^2$SA guarantee becomes $\widetilde O(\eps^{-7})$ \citep{kwon2023}, and why the $\widetilde O(\varepsilon^{-3/2})$ guarantee of FdeHBO becomes $\widetilde O(\eps^{-3})$ under \emph{its stronger sample assumptions} \citep{yang2023fde}. The reverse implication from expected norm to expected squared norm does not hold.

\paragraph{The 2023 and 2025 fully first-order papers.}
The original F$^2$SA analysis includes a full Hessian-Lipschitz condition on $g$ and additional assumptions in its principal convergence result \citep{kwon2023}. The F$^2$BA framework explicitly assumes full Hessian Lipschitzness of $g$ in Assumption~3.1(c) \citep{chen2025}. Our main model requires joint Lipschitzness of $g_{xy}$ and $g_{yy}$ but not of $g_{xx}$. The sine lift's pure $x$ derivatives need not obey the stronger normalization. Their upper bounds motivate the precision question. Our matching result concerns the global two-block class with possibly noisy lower gradients; their deterministic, upper-only-noisy, and second-order-stationarity guarantees address additional regimes.

\paragraph{The finite-order 2026 framework.}
Assumptions~2.1--2.4 of \citet{chen2026} give the same kind of baseline conditions and globally unbiased bounded-variance gradients. In Assumption~2.5, their bound on $D_y^q\nabla f$ is stated jointly, while the $g$-derivative bound is in $y$. Our upper proof is valid under the displayed $y$-only assumptions, and the lower $f$ has vanishing such derivatives; imposing their stronger version on $f$ therefore preserves the matching result. The full-text Theorem~3.1 has the bound
\[
 O\!\left(
 p\Delta L_1\bar L^{2/p}\sigma^2\kappa^{9+2/p}
 \eps^{-4-2/p}\log\!\left(\frac{RL_1\bar L\kappa}{\eps}\right)
 \right)
\]
in its notation and nondegenerate parameter regime. We use the precision exponent $4+2/p$ stated in that theorem. Its high-order corollary gives a nearly $\eps^{-4}$ rate. Here $\bar L$ is a maximum over derivative budgets up to the chosen order, and $\kappa=\bar L/\mu$. Substituting $L_j=M\Lambda^j(j!)^\beta$ makes their dependence on the chosen order explicit. Our uniform theorem analyzes that derivative growth directly, rather than reinterpreting a fixed-order hidden constant.

\subsection{Local unbiasedness is a different observation contract}
Definition~1.1 of \citet{kwon2024} introduces a $y^*$-aware oracle: a call includes an approximate lower solution, and gradient unbiasedness is required only within a prescribed neighborhood of the true lower solution. Their lower bound uses a finite accuracy-dependent radius. Outside that neighborhood an oracle can supply uninformative responses not equal in expectation to the true gradient. Their upper method also works for the ordinary globally unbiased oracle (the infinite-radius case), establishing the existing $O(\eps^{-6})$ baseline upper bound. Their expected-complexity convention and oracle accounting differ from our fixed-budget formulation. The lower bounds therefore apply to distinct observation contracts.

In our lower instance, the mean of the full vector $G_g$ is $\nabla g$ at every real $(x,y)$, and its conditional variance is uniformly bounded. There is no accuracy-dependent region outside which the mean guarantee is suspended. The bounded sine signal, not a permitted oracle bias, prevents far-away $y$ queries from bypassing the information lower bound.

\subsection{Adjoint-based algorithms and sample reuse}
FdeHBO \citep{yang2023fde} is an important methodological antecedent to the projected adjoint in~\eqref{eq:main-adjoint}. Its equation~(9) differences two gradients with the \emph{same} sample at $y\pm\delta v$, and equations~(10)--(12) use shared-sample momentum differences. Its Assumption~3 imposes samplewise versions of its smoothness conditions and second-moment bounds on stochastic Hessian/Jacobian blocks. Thus ``Hessian-free computation'' there does not mean that only population smoothness and one-point fresh samples are assumed. Our claims concern the latter contract. The contributions of the upper proof are the uniform Chebyshev weight estimate, variance accounting under independent node samples, and the absence of multiplicative order and contraction factors in the final bound. The polynomial differentiation component builds on classical numerical differentiation \citep{fornberg1988}.

SPABA \citep{chu2024} addresses expectation and finite-sum formulations with variance reduction and a different derivative/sampling interface. SGHA \citep{gu2026} obtains baseline stochastic exponent six and an improved exponent under additional lower-level stochastic smoothness; its probability and bounded-iterate qualifications are also not identical to~\eqref{eq:target}. Their stronger sampling contracts explain why the improved exponents coexist with the lower bounds for fresh one-point samples studied here.

\subsection{Condition-number lower bounds and their versions}
The earlier first-order lower-bound work of \citet{ji2026} was subsequently merged and strengthened in \citet{chenji2026}. The latter's version~4, Theorem~5.2, gives
\[
 \Omega\!\left(\kappa_y^{5/2}L_1\Delta\eps^{-2}
       +\kappa_y^{9/2}\sigma^2L_1\Delta\eps^{-4}\right)
\]
for its zero-respecting SFO class, constrained domains, and projected-gradient mapping criterion. It improves condition-number factors while retaining a precision question for standard stochastic gradients. The explicit condition-number factors address a complementary complexity question to the order--accuracy tradeoff studied here. Our lower construction treats arbitrary randomized algorithms and growing upper dimension, under the unconstrained norm criterion and fixed regularity parameters.

\section{Experimental Constructions and Implementation}\label{app:experiments}
This appendix records the numerical choices behind Section~\ref{sec:experiments}. The lower construction uses the proof's rare-revelation oracle. The fixed-collection experiment instead uses independent additive Gaussian gradient noise. Both count complete first-order samples and evaluate the same uniform-output gradient-norm criterion.

\subsection{Numerical lower-bound construction}
We fix $(\mu,L_0,L_1,L_2,\Delta,\sigma^2)=(1,1,4,1,10^{-6},1)$, $R_0=0$, and the higher-derivative envelope $M=\Lambda=1$, $\beta=0$. For the chain and wrapper estimates we use conservative bounds
\[
 C_B=29,\quad C_{BB}=5000,\quad C_{AA}=52000,\quad
 C_H=1536,\quad C_\Delta=93,\quad B_0=1/4.
\]
Set $K_0=10^{10}$, $\tau=10^{-15}$, $b=1/2$, and $m=5/2$. For each $N$ and $p$, a geometric bisection chooses the largest feasible $s$ in this fixed parametrization, with a relative numerical margin of $10^{-10}$, subject to
\begin{align*}
 h&=K_0s^{1/p},& T&=\left\lfloor\frac{\Delta\tau}{2C_\Delta s^2}\right\rfloor\ge2,\\
 q_s&=C_*^2s^2h^2<1,&8Nq_s&\le T,
 \qquad C_*=(m/b)(C_B+B_0/(\tau K_0)).
\end{align*}
The sine lift has amplitude $s^2/\tau$ and upper-coordinate scale $s/\tau$. Each configuration is checked against the initial-gap, strong-convexity, full-gradient smoothness, block regularity, higher-derivative, implicit-root contraction, and revelation-budget inequalities. This maximizes a certified barrier within the specified parametrization, not over every admissible instance.

For $\delta=1/16$ and $R=64C_H\sqrt T$, the ambient dimension is
\[
 d_x=T+\left\lceil1.01\frac{2R^2}{\delta^2}
                    \log(16(N+1)T)\right\rceil.
\]
The tested methods start at zero and use rotation-equivariant operations. Their trajectories can therefore be represented exactly in latent coordinates without storing this ambient frame: rotating the instance rotates the trajectories and preserves their error norms. Coordinates are stored lazily without imposing a limit on newly activated directions. This is a computational reduction for the tested methods, not an empirical test of every unrestricted algorithm covered by the theorem.

The theoretical upper oracle is noiseless, and its numerical evaluation is charged. At every lower query the oracle draws an independent Bernoulli variable as in~\eqref{eq:main-rare}. A batch of $B$ samples at one point is simulated by a binomial count and charged $B$ calls. During consecutive nonrevealing lower samples, the inner recursion is affine; geometric waiting times and affine propagation skip these events computationally while charging every sample. At revealing events the nonlinear sine-lift gradient is evaluated without linearization. Consequently budgets up to $10^{16}$ denote simulated oracle samples, not that many explicit software calls. The event-skipping recursion is checked against direct recursion with the same Bernoulli trace.

The scalar primitive $r$ is evaluated by cumulative trapezoidal quadrature on 200001 equally spaced points in $[-2,2]$, followed by linear interpolation; its derivative and the other scalar fields use their explicit formulas. An independent split Gauss--Legendre check at 1401 points gives a maximum primitive discrepancy of $6.1\cdot10^{-8}$, with the 64- and 96-node references agreeing within $1.5\cdot10^{-14}$. This is a numerical evaluation check, separate from the exact-law batch and event simulation. The figures use floating-point evaluations of the theoretical construction rather than an exact-arithmetic oracle implementation.

For the budget-dependent panels, Ours uses the explicit tolerances and geometric allocation in Appendices~\ref{app:differencing} and~\ref{app:upper}, including $b_j=\lceil B|d_j|/W\rceil$. Each method's design tolerance is determined from its complete cost before observing errors so that the prescribed schedule fits $N$. The design tolerance is distinct from the target $s_N/16$. For the baselines, source scaling relations are instantiated with explicit constants; constants hidden in asymptotic source statements are not certified by this experiment. The 17 budgets, two orders, four methods, and eight noise seeds give 1088 trajectories.

For the fixed-target panels, the $p=2$ instance is constructed at reference budget $10^5$, giving $T=3$. For $p=4$, we construct the scale at budget $100$ and shorten the chain to $T=3$ while retaining its scale and oracle; the resulting certified revelation budget is four. Shortening preserves the upper bound on the initial gap. The dimension is fixed using the common horizon $10^{20}$. The targets are those printed in Figure~\ref{fig:lower-experiment} and differ between its two rows. All plotted budgets in one panel extend the same trajectories on the same instance. Each lower and adjoint solve uses two steps of length $0.4$, with difference radius or perturbation one; Kwon uses penalty parameter one. These choices change the source sampling schedules and Kwon's inner step rule.

Writing $w=x/\lambda$ for the construction's upper scale, the normalized methods use outer step $0.05$ for the first 2000 steps and then $0.05[1+(t-2000)/10]^{-3/2}$. Kwon uses the nonnormalized step $0.0005$ in $w$. The base batch is $b_{\rm base}=\lceil10^5/q_s\rceil$. Ours uses projection radius one and weighted node batches: with $S_2=\sum_jd_j^2$, set $B=\lceil b_{\rm base}W^2/S_2\rceil$ and $b_j=\lceil B|d_j|/W\rceil$. This preserves the variance upper bound of the corresponding equal-node allocation. All plotted Ours results use this weighted version. The other methods retain independent inner batches. The normalized methods run for $10^5$ outer steps; the unchanged Kwon trajectories were extended to $2\cdot10^5$ steps to reach the common error target. These short-chain experiments illustrate eventual passage below the barrier, not the asymptotic upper-bound schedules.

\subsection{The fixed collection}\label{app:experiment-family}
Let $x,y\in\R^4$, $z=y-Ax$, and $m_0=8$. For $L\in\{4,8,16\}$ define
\begin{align}
 g_L(x,y)&=\tfrac12\norm z^2+
     \frac{\rho}{m_0}\sum_{i=1}^{m_0}
         [1-\cos(c_{x,i}^{\mathsf T}x+c_{y,i}^{\mathsf T}y+\phi_i)],\label{eq:exp-g}\\
 f_L(x,y)&=\arctan(x_1/L)+\pi/2+b^{\mathsf T}z
       +\frac1{m_0}\sum_{i=1}^{m_0}[1-\cos(u_i^{\mathsf T}x+\theta_i)]\notag\\*
 &\hspace{1cm}+\frac{\tau_f}{m_0}\sum_{i=1}^{m_0}
         [1-\cos(e_{x,i}^{\mathsf T}x+e_{y,i}^{\mathsf T}y+\psi_i)],\label{eq:exp-f}
\end{align}
where $\rho=0.45$ and $\tau_f=0.2$. The initial matrix is $A=I/2$, $b$ is a normalized Gaussian vector, and the rows of $C=(C_x,C_y)$, $E=(E_x,E_y)$, and $U$ are independent normalized Gaussian vectors. The three phase vectors are uniform on $[0.3,2.3]^8$. We use one fixed coefficient draw (NumPy default generator, seed 1303), then zero the first upper-coordinate column of $A,C,E,U$. All three members share these coefficients, and both variables start at zero. Saved coefficient arrays make the instances independent of the generator version.

The unit row norms and $\norm A\le1/2$ give uniform budgets
\[
 \mu=0.5,\quad L_0=1.25,\quad L_1=2,\quad L_2=0.5,
 \quad \Delta=8,\quad R_0=0.5,\quad M=\Lambda=1,\quad\beta=0.
\]
For example, $g_{yy}\succeq(1-\rho)I$, and the higher lower-variable derivatives of the trigonometric terms are bounded uniformly in their order. The added arctangent is independent of $y$ and has range $\pi$, so it does not consume higher lower-variable derivative budgets. Its first-coordinate gradient is exactly $[L(1+(x_1/L)^2)]^{-1}$. Each full gradient observation adds fresh Gaussian noise with covariance $(\sigma^2/8)I_8$, where $\sigma=0.2$. Independent batch means are sampled from their exact Gaussian law, with all constituent calls charged.

\subsection{Frozen practical schedules for the rate experiment}
We retain the update structures of Algorithm~\ref{alg:uniform}, F$^2$SA-$p$, and the standard globally unbiased branch of Kwon's algorithm. The following constants are practical choices rather than a certification of all sufficient theorem constants. The query radii scale as
\[
 h(\eps)=h_{\rm ref}(\eps/0.04)^{1/p},\qquad
 \nu(\eps)=\nu_{\rm ref}(\eps/0.04)^{1/p}.
\]
The pairs $(h_{\rm ref},\nu_{\rm ref})$ are approximately $(0.718707,0.632932)$ at $p=2$ and $(1.644934,1.131744)$ at $p=4$. They come from a preceding calibration on separate coefficient draws and are frozen for the present collection; no radius is retuned at an observed accuracy point. That finite calibration is not a global bias certificate.

\paragraph{Ours.}
Use $J=\lceil\log(25/\eps)/[-\log(0.875)]\rceil$ steps in each restarted solve, lower step $0.4$, adjoint step $0.25$, and adjoint projection radius $2.5$. Set
\[
 B_y=\lceil2\sigma^2/\eps^2\rceil,\qquad
 B_v=B_o=\lceil4\sigma^2(1+W^2)/\eps^2\rceil.
\]
At inner step $k=0,\ldots,J-1$, use $\max\{1,\lceil B_y2^{-(J-1-k)/2}\rceil\}$ lower samples, and the analogous $B_v$ schedule for the adjoint. Every difference uses the paper's independent weighted allocation. The final estimate uses fresh samples, and the normalized outer step is $\eps/4$.

\paragraph{F$^2$SA-$p$.}
For even $p=2m$, use nonzero central nodes $j\in\{-m,\ldots,m\}\setminus\{0\}$ and weights
\[
 a_j=\frac{(-1)^{j+1}(m!)^2}{j(m-j)!(m+j)!}\quad(j>0),\qquad a_{-j}=-a_j.
\]
The warm-started inner update uses independent samples of $g_y+j\nu f_y$, step $\alpha_y=\min\{0.4,(\nu\eps/\sigma)^2/2\}$, and length $K=\lceil\log(2/(\nu\eps))/\alpha_y\rceil$. Every single-sample recursion is executed; it is not replaced by one batched update. For the final combination $\sum_j a_j(j\overline f_x+\overline g_x/\nu)$, each node uses
\[
 S=\left\lceil\frac{4\sigma^2}{\eps^2}
                \sum_j a_j^2(j^2+\nu^{-2})\right\rceil
\]
fresh samples of each function. The normalized outer step is $\eps/4$, and the per-outer cost is $2p(K+S)$. Perturbed lower strong convexity is checked for each configuration.

\paragraph{Kwon.}
Use penalty $\lambda_K=0.2/\eps$, inner length $T_K=\lceil4\sigma^2\lambda_K^2/\eps^2\rceil$, and final batch $M_K=\lceil4\sigma^2(1+2\lambda_K^2)/\eps^2\rceil$. Two warm-started inner trajectories use gradients $g_y+f_y/\lambda_K$ and $g_y$, with steps
\[
 \gamma_t=\frac{4+\lambda_K/(2+2\lambda_K)}{1+t},\qquad t=0,\ldots,T_K-1.
\]
This denominator follows the displayed schedule in \citet[Theorem~3.1]{kwon2024}; their proof uses the shifted denominator $1+\sqrt3+t$ in equation~(37). We specify the implemented convention here to make the comparison reproducible. The nonnormalized outer step is $0.25$, using the independent estimate of $f_x+\lambda_K(g_x(y)-g_x(z))$. The per-outer cost is $3T_K+3M_K$. Projection is the identity in the globally unbiased, unrestricted-query branch. No reusable-sample smoothness guarantee is invoked.

\subsection{Evaluation, diagnostics, and absolute costs}
The evaluator solves the four-dimensional lower problem by Newton iteration and computes the implicit hypergradient; it provides no search information. A directional finite-difference check has error below $2\cdot10^{-11}$, and the compiled gradient and noiseless Kwon recursion agree with their direct versions within $3\cdot10^{-16}$. Budgets are inspected on a grid growing approximately by a factor $1.15$. Thus $N_A$ is a recorded successful budget, not an exact continuous hitting time. Main runs use eight noise trajectories; an independent set repeats levels $0,3,6$. The repetitions assess noise sensitivity while keeping the problems fixed.

At the finest target $\eps=0.03145728$, the maximum recorded costs are
\begin{center}
\begin{tabular}{@{}lrr@{}}
\toprule
Method & $p=2$ & $p=4$\\
\midrule
Ours & $295{,}958{,}534$ & $149{,}191{,}742$\\
F$^2$SA-$p$ & $54{,}742{,}144$ & $40{,}421{,}712$\\
Kwon & $260{,}015{,}040$ & $260{,}015{,}040$\\
\bottomrule
\end{tabular}
\end{center}
Kwon has no order parameter, so its data are reused in both panels. The observed slow-coordinate contribution accounts for approximately $79$--$85\%$ of the uniform-output error for Ours and F$^2$SA-$p$, and $91$--$95\%$ for Kwon. Its effect is therefore visible separately from the sample schedules. For Ours, node and batch rounding account for approximately $23\%$ ($p=2$) and $31\%$ ($p=4$) of coarse-target per-outer costs, decreasing to about $0.7\%$ and $2.0\%$ at the finest target. These effects help explain why the measured slopes need not equal their leading theoretical exponents within this finite window.

The plotting sources retain per-member results, maximizing identities (including ties), absolute costs, source hashes, adjacent and nested-range slopes, and independent-noise repetitions. The measured ratios divide by the actual initial observed cost, whereas the theory guides fit only one multiplier per method using all seven points. The logarithmic factor and theoretical exponent remain fixed throughout.

\end{document}